\documentclass[reqno]{amsart}
\usepackage{amssymb}
\usepackage{graphicx}

\usepackage[usenames, dvipsnames]{color}
\usepackage{verbatim}
\usepackage{mathrsfs}
\usepackage{bm}
\usepackage{cite}

\numberwithin{equation}{section}

\newtheorem{theorem}{Theorem}[section]
\newtheorem{corollary}[theorem]{Corollary}
\newtheorem{lemma}[theorem]{Lemma}

\theoremstyle{definition}
\newtheorem{remark}[theorem]{Remark}

\theoremstyle{definition}
\newtheorem{definition}[theorem]{Definition}

\theoremstyle{definition}

\makeatletter
\def\dashint{\operatorname%
{\,\,\text{\bf-}\kern-.98em\DOTSI\intop\ilimits@\!\!}}
\makeatother

\def\\det{\text{\det}}

\def\Xint#1{\mathchoice
 {\XXint\displaystyle\textstyle{#1}}%
 {\XXint\textstyle\scriptstyle{#1}}%
 {\XXint\scriptstyle\scriptscriptstyle{#1}}%
 {\XXint\scriptscriptstyle\scriptscriptstyle{#1}}%
 \!\int}
\def\XXint#1#2#3{{\setbox0=\hbox{$#1{#2#3}{\int}$}
  \vcenter{\hbox{$#2#3$}}\kern-.5\wd0}}

\def\dashint{\Xint-}

\def\.5{\frac{1}{2}}

\newcommand{\RN}[1]{%
  \textup{\uppercase\expandafter{\romannumeral#1}}%
}

\renewcommand{\epsilon}{\varepsilon}

\newcounter{marnote}

\begin{document}

\title[Anisotropic Muckenhoupt weights and their applications]{On a class of anisotropic Muckenhoupt weights and their applications to $p$-Laplace equations}

\author[C.X. Miao]{Changxing Miao}
\address[C.X. Miao] {Institute of Applied Physics and Computational Mathematics, P.O. Box 8009, Beijing, 100088, China.}
\email{miao\_changxing@iapcm.ac.cn}

\author[Z.W. Zhao]{Zhiwen Zhao}

\address[Z.W. Zhao]{1. School of Mathematics and Physics, University of Science and Technology Beijing, Beijing 100083, China.}
\address{2. Beijing Computational Science Research Center, Beijing 100193, China.}
\email{zwzhao365@163.com}


\date{\today} 


\maketitle
\begin{abstract}
In this paper, a class of anisotropic weights having the form of $|x'|^{\theta_{1}}|x|^{\theta_{2}}|x_{n}|^{\theta_{3}}$ in dimensions $n\geq2$ is considered, where $x=(x',x_{n})$ and $x'=(x_{1},...,x_{n-1})$. We first find the optimal range of $(\theta_{1},\theta_{2},\theta_{3})$ such that this type of weights belongs to the Muckenhoupt class $A_{p}$. Then we further study its doubling property, which shows that it provides an example of a doubling measure but is not in $A_{p}$. As a consequence, we obtain anisotropic weighted Poincar\'{e} and Sobolev inequalities, which are used to study the local behavior for solutions to non-homogeneous weighted $p$-Laplace equations.
\end{abstract}

\section{Introduction}

Nowdays, there has been a surge of interest in studying a class of anisotropic weights comprising of different power-type weights due to their extensive and significant applications in composite materials, fluid mechanics and diffusion problems of the flows etc. To be specific, for $n\geq2$ and $\theta_{i}\in\mathbb{R}$, $i=1,2,3,$ the weight considered in this paper has the form of $|x'|^{\theta_{1}}|x|^{\theta_{2}}|x_{n}|^{\theta_{3}}$. Here and throughout this paper, we denote the $(n-1)$-dimensional variables and domains by adding superscript prime, such as $x'=(x_{1},...,x_{n-1})$ and $B_{1}'(0)$. Note that these three component parts in the weight are of different dimensions if $n\geq3$. Furthermore, different from the fact that $|x|^{\theta_{2}}$ exhibits singular or degenerate behavior only at the origin, $|x'|^{\theta_{1}}$ and $|x_{n}|^{\theta_{3}}$ represent singular or degenerate line and plane, respectively.

When $\theta_{2}\neq0$ and $\theta_{1}=\theta_{3}=0$, $|x|^{\theta_{2}}$ is a well-known isotropic weight due to the famous work \cite{CKN1984}, where Caffarelli, Kohn and Nirenberg \cite{CKN1984} established the following interpolation inequalities
\begin{align}\label{CKN90}
\||x|^{\gamma_{1}}u\|_{L^{s}(\mathbb{R}^{n})}\leq C\||x|^{\gamma_{2}}\nabla u\|^{a}_{L^{p}(\mathbb{R}^{n})}\||x|^{\gamma_{3}}u\|^{1-a}_{L^{q}(\mathbb{R}^{n})},\quad\text{for any }u\in C^{\infty}_{0}(\mathbb{R}^{n}),
\end{align}
and found the necessary and sufficient condition such that it holds under natural assumptions on the parameters. The Caffarelli-Kohn-Nirenberg weight has been widely applied to the study on weighted Sobolev spaces and relevant PDEs. In particular, Dong, Li and Yang \cite{DLY2022,DLY2021} recently revealed the optimal gradient blow-up rate for solution to the linear insulated conductivity problem arising from composite materials by precisely capturing the asymptotic behavior of solution to the weighted elliptic equation as follows:
\begin{align*}
\mathrm{div}(\kappa(x)|x|^{2}\nabla u)=0,\quad\text{in }B_{1},
\end{align*}
where $\tau^{-1}\leq\kappa\leq\tau$ in $B_{1}$ for some $\tau>0$, and $\int_{\mathbb{S}^{n-1}}\kappa x_{i}=0$, $i=1,2,...,n$. Under these assumed conditions, they made use of spherical harmonic expansion to obtain
\begin{align*}
u(x)=u(0)+O(1)|x|^{\alpha},\quad\alpha=2^{-1}\big(-n+\sqrt{n^{2}+4\lambda_{1}}\big),\quad\mathrm{in}\;B_{1/2}.
\end{align*}
Here $\lambda_{1}\leq n-1$ represents the first nonzero eigenvalue for the following eigenvalue problem:
$-\mathrm{div}_{\mathbb{S}^{n-1}}(\kappa(\xi)\nabla_{\mathbb{S}^{n-1}}u(\xi))=\lambda \kappa(\xi) u(\xi)$, $\xi\in\mathbb{S}^{n-1}.$
Especially when $\kappa=1$, we have $\lambda_{1}=n-1$. See Lemma 2.2 in \cite{DLY2021} and Lemmas 2.2 and 5.1 in \cite{DLY2022} for more details. Before that, the local behavior for the solution to weighted elliptic problem has been investigated in the earlier work \cite{FKS1982}. However, the value of H\"{o}lder exponent $\alpha$ captured in \cite{FKS1982} is not explicit. In fluid mechanics, this weight was also utilized to describe the singular behavior of Landau solutions \cite{L1944} with the singularity of order $O(|x|^{-1})$.

When $\theta_{1}\neq0$, $\theta_{2}\in\mathbb{R}$ and $\theta_{3}=0$, the weight $|x'|^{\theta_{1}}|x|^{\theta_{2}}$ is anisotropic. This type of weights has been used in \cite{LY2021} to classify the singularities of $(-1)$-homogeneous axisymmetric no-swirl solutions to the three-dimensional stationary Navier-Stokes equations found in \cite{LLY201801,LLY201802}. Furthermore, Li and Yan \cite{LY2021} studied the asymptotic stability of the singular solutions of order $O(|x|^{-1}\ln|x'|^{-1}|x|)$ to non-stationary Navier-Stokes equations. The key to their proof lies in establishing an extended Hardy-type inequality with anisotropic weights, see Theorem 1.3 in \cite{LY2021}. Stimulated by this problem, Li and Yan \cite{LY2023} recently extended the Caffarelli-Kohn-Nirenberg inequalities in \eqref{CKN90} to the anisotropic case.

When $\theta_{1}=\theta_{2}=0$ and $\theta_{3}\neq0$, the weight $|x_{n}|^{\theta_{3}}$ is closely linked to the establishment for the global regularity for diffusion equations. To be specific, Jin and Xiong \cite{JX2019} established the boundary estimates for the linearized weighted fast diffusion equation such that the implicit function theorem can be utilized to produce a short time regular solution $u$ of fast diffusion equation. Based on the obtained short time regular solution, they further bootstrap its regularity to be smooth provided that the H\"{o}lder continuity of $\partial_{t}u/u$ is proved. This difficulty was overcame in \cite{JX2019} by proving arbitrarily high integrability of $\partial_{t}u/u$. Different from this idea, Jin and Xiong \cite{JX2022} reduced this difficulty to establishing the H\"{o}lder estimates of $u/x_{n}$ in $B_{1/2}^{+}\times(-1/4,0]$, which is achieved by studying the H\"{o}lder regularity for a weighted nonlinear parabolic equation. Their regularity results are optimal, which answers a open problem proposed by Berryman and Holland \cite{BH1980}. Subsequently, Jin, Ros-Oton and Xiong \cite{JRX2023} further extended the results to the case of porous medium equations. With regard to more investigations on the regularity for PDEs with the weight $x_{n}$, we refer to \cite{DPT2023,FP2013,DP2023,DPT202302,JX2023} and the references therein.

In order to better apply the weight $|x'|^{\theta_{1}}|x|^{\theta_{2}}|x_{n}|^{\theta_{3}}$ to geometric and functional inequalities and relevant PDEs, we need first make clear its own properties, especially finding the largest range of $(\theta_{1},\theta_{2},\theta_{3})$ such that it becomes an $A_{p}$-weight. In fact, $A_{p}$-weights play a significant role in fourier analysis and nonlinear potential theory, see e.g. \cite{HKM2006,G2014} for systematic investigations on the Muckenhoupt class $A_{p}$.

The rest of the paper is organized as follows. Section \ref{SEC02} is dedicated to the finding of optimal range of $(\theta_{1},\theta_{2},\theta_{3})$ which makes $|x'|^{\theta_{1}}|x|^{\theta_{2}}|x_{n}|^{\theta_{3}}$ be in the Muckenhoupt class $A_{p}$. Furthermore, its doubling property is proved to be satisfied in a larger range. This implies that the weight $|x'|^{\theta_{1}}|x|^{\theta_{2}}|x_{n}|^{\theta_{3}}$ provides an example of a doubling measure but is not of the Muckenhoupt class $A_{p}$. As a direct application of the theories of $A_{p}$-weights, we obtain the anisotropic weighted Sobolev and Poincar\'{e} inequalities, which are subsequently used to investigate the local behavior for solutions to non-homogeneous $p$-Laplace equations with anisotropic weights in Section \ref{SEC03}. Finally, we further discuss their applications in doubly non-linear parabolic equations and some intriguing problems are naturally raised in Section \ref{SEC05}.

\section{Anisotropic Muckenhoupt weights consisting of three different power-type weights}\label{SEC02}
For $p>1$, define
\begin{align*}
\begin{cases}
\mathcal{A}=\{(\theta_{1},\theta_{2},\theta_{3}): \theta_{1}>-(n-1),\,\theta_{2}\geq0,\,\theta_{3}>-1\},\\
\mathcal{B}=\{(\theta_{1},\theta_{2},\theta_{3}):\theta_{1}>-(n-1),\,\theta_{2}<0,\,\theta_{3}>-1,\,\theta_{1}+\theta_{2}+\theta_{3}>-n\},\\
\mathcal{C}_{p}=\{(\theta_{1},\theta_{2},\theta_{3}):\theta_{1}<(n-1)(p-1),\,\theta_{2}\leq0,\,\theta_{3}<p-1\},\\
\mathcal{D}_{p}=\{(\theta_{1},\theta_{2},\theta_{3}):\theta_{1}<(n-1)(p-1),\,\theta_{2}>0,\,\theta_{3}<p-1,\,\theta_{1}+\theta_{2}+\theta_{3}<n(p-1)\}.
\end{cases}
\end{align*}
Introduce the definition of $A_{p}$-weight as follows.
\begin{definition}
Set $1<p<\infty$. We say that $w$ is an $A_{p}$-weight (or $w$ is of the Muckenhoupt class $A_{p}$), provided that there exists some constant $C=C(n,p,w)>0$ such that
\begin{align*}
\dashint_{B}wdx\left(\dashint_{B}w^{-\frac{1}{p-1}}dx\right)^{p-1}\leq C(n,p,w),\quad\mathrm{with}\;\dashint_{B}=\frac{1}{|B|}\int_{B},
\end{align*}
for any ball $B$ in $\mathbb{R}^{n}$.

\end{definition}

In the following, we use $\omega_{n}$ to represent the volume of unit ball in $\mathbb{R}^{n}$. The notation $a\sim b$ implies that $\frac{1}{C}b\leq a\leq Cb$ for some constant $C=C(n,\theta_{1},\theta_{2},\theta_{3})>0$. The first step is to find the largest range of $(\theta_{1},\theta_{2},\theta_{3})$ ensuring that the weight $|x'|^{\theta_{1}}|x|^{\theta_{2}}|x_{n}|^{\theta_{3}}$ is a Radon measure, which is a necessary condition such that it further becomes an $A_{p}$-weight.
\begin{lemma}\label{MWQAZ090}
$d\mu:=wdx=|x'|^{\theta_{1}}|x|^{\theta_{2}}|x_{n}|^{\theta_{3}}dx$ becomes a Radon measure iff $(\theta_{1},\theta_{2},\theta_{3})\in\mathcal{A}\cup\mathcal{B}$. Furthermore, $\mu(B_{r})\sim r^{n+\theta_{1}+\theta_{2}+\theta_{3}}$ for any $r>0$.

\end{lemma}

%

\begin{proof}

\noindent{\bf Step 1.} Let $\theta_{2}\geq0$. Then we obtain
\begin{align*}
\mu(B_{r})=&2\int_{B_{r}\cap\{x_{n}>0\}}|x'|^{\theta_{1}}|x|^{\theta_{2}}x_{n}^{\theta_{3}}dx\notag\\
\sim&\int_{B_{r}'}|x'|^{\theta_{1}}dx'\int_{0}^{\sqrt{r^{2}-|x'|^{2}}}(|x'|^{\theta_{2}}+x_{n}^{\theta_{2}})x_{n}^{\theta_{3}}dx_{n}\notag\\
\sim&\int^{r}_{0}\left(s^{n+\theta_{1}+\theta_{2}-2}(r^{2}-s^{2})^{\frac{\theta_{3}+1}{2}}+s^{n+\theta_{1}-2}(r^{2}-s^{2})^{\frac{\theta_{2}+\theta_{3}+1}{2}}\right)ds\notag\\
\sim&\,r^{n+\theta_{1}+\theta_{2}+\theta_{3}}\int^{1}_{0}\left(s^{n+\theta_{1}+\theta_{2}-2}(1-s)^{\frac{\theta_{3}+1}{2}}+s^{n+\theta_{1}-2}(1-s)^{\frac{\theta_{2}+\theta_{3}+1}{2}}\right)ds,
\end{align*}
where in the second line we require $\theta_{3}>-1$ and utilized the following elemental inequalities: for $a,b\geq0$,
\begin{align}\label{INE001}
\begin{cases}
\frac{a^{p}+b^{p}}{2}\leq(a+b)^{p}\leq a^{p}+b^{p},&0\leq p\leq1,\\
a^{p}+b^{p}\leq(a+b)^{p}\leq2^{p-1}(a^{p}+b^{p}),&p>1.
\end{cases}
\end{align}
Combining the properties of Beta function, we see that this integration makes sense iff $n+\theta_{1}+\theta_{2}>1$, $n+\theta_{1}>1$ and $\theta_{3}>-1$. Hence the conclusion holds under the condition of $(\theta_{1},\theta_{2},\theta_{3})\in\mathcal{A}$.

\noindent{\bf Step 2.} Let $\theta_{2}<0$. From \eqref{INE001}, we have
\begin{align}\label{WDZM001}
\mu(B_{r})=&2\int_{B_{r}\cap\{x_{n}>0\}}|x'|^{\theta_{1}}|x|^{\theta_{2}}x_{n}^{\theta_{3}}dx\notag\\
\sim&\int_{B_{r}'}|x'|^{\theta_{1}}dx'\int^{\sqrt{r^{2}-|x'|^{2}}}_{0}\frac{x_{n}^{\theta_{3}}}{|x'|^{-\theta_{2}}+x_{n}^{-\theta_{2}}}dx_{n}.
\end{align}
The last integration term in \eqref{WDZM001} can be decomposed as follows:
\begin{align*}
I_{1}=&\int_{B'_{\frac{1}{\sqrt{2}}r}}|x'|^{\theta_{1}}dx'\int^{|x'|}_{0}\frac{x_{n}^{\theta_{3}}}{|x'|^{-\theta_{2}}+x_{n}^{-\theta_{2}}}dx_{n},\notag\\
I_{2}=&\int_{B'_{\frac{1}{\sqrt{2}}r}}|x'|^{\theta_{1}}dx'\int^{\sqrt{r^{2}-|x'|^{2}}}_{|x'|}\frac{x_{n}^{\theta_{3}}}{|x'|^{-\theta_{2}}+x_{n}^{-\theta_{2}}}dx_{n},\notag\\
I_{3}=&\int_{B_{r}'\setminus B'_{\frac{1}{\sqrt{2}}r}}|x'|^{\theta_{1}}dx'\int^{\sqrt{r^{2}-|x'|^{2}}}_{0}\frac{x_{n}^{\theta_{3}}}{|x'|^{-\theta_{2}}+x_{n}^{-\theta_{2}}}dx_{n}.
\end{align*}
As for the first term $I_{1}$, since $|x'|^{-\theta_{2}}\leq|x'|^{-\theta_{2}}+x_{n}^{-\theta_{2}}\leq2|x'|^{-\theta_{2}}$ if $0\leq x_{n}\leq |x'|$, we have
\begin{align*}
I_{1}\sim&\int_{B'_{\frac{1}{\sqrt{2}}r}}|x'|^{\theta_{1}+\theta_{2}}dx'\int^{|x'|}_{0}x_{n}^{\theta_{3}}dx_{n}\sim\int^{\frac{r}{\sqrt{2}}}_{0}s^{n+\theta_{1}+\theta_{2}+\theta_{3}-1}ds\notag\\
\sim&r^{n+\theta_{1}+\theta_{2}+\theta_{3}}\int^{1}_{0}s^{n+\theta_{1}+\theta_{2}+\theta_{3}-1}ds.
\end{align*}
This integration is valid iff $n+\theta_{1}+\theta_{2}+\theta_{3}>0$ and $\theta_{3}>-1$.

For the second term $I_{2}$, in the following we divide into two subcases to discuss.

{\bf Case 1.} When $\theta_{2}+\theta_{3}=-1$, we have
\begin{align*}
I_{2}\sim&\int_{B'_{\frac{1}{\sqrt{2}}r}}|x'|^{\theta_{1}}dx'\int^{\sqrt{r^{2}-|x'|^{2}}}_{|x'|}x_{n}^{-1}dx_{n}\sim\int^{\frac{r}{\sqrt{2}}}_{0}s^{n+\theta_{1}-2}\ln\frac{\sqrt{r^{2}-s^{2}}}{s}ds\notag\\
\sim&\,r^{n+\theta_{1}-1}\int^{\frac{1}{\sqrt{2}}}_{0}s^{n+\theta_{1}-2}\ln\frac{\sqrt{1-s^{2}}}{s}ds\sim-r^{n+\theta_{1}-1}\int^{\frac{1}{\sqrt{2}}}_{0}s^{n+\theta_{1}-2}\ln s\,ds\notag\\
\sim&-r^{n+\theta_{1}+\theta_{2}+\theta_{3}}\bigg(s^{n+\theta_{1}+\theta_{2}+\theta_{3}}\ln s\Big|^{\frac{1}{\sqrt{2}}}_{0}-\int^{\frac{1}{\sqrt{2}}}_{0}s^{n+\theta_{1}+\theta_{2}+\theta_{3}-1}ds\bigg).
\end{align*}
This is valid iff $n+\theta_{1}+\theta_{2}+\theta_{3}>0$.

{\bf Case 2.} When $\theta_{2}+\theta_{3}\neq-1$, we have
\begin{align*}
I_{2}\sim&\int_{B'_{\frac{1}{\sqrt{2}}r}}|x'|^{\theta_{1}}dx'\int^{\sqrt{r^{2}-|x'|^{2}}}_{|x'|}x_{n}^{\theta_{2}+\theta_{3}}dx_{n}\notag\\
\sim&\int^{\frac{1}{\sqrt{2}}r}_{0}s^{n+\theta_{1}-2}\left((r^{2}-s^{2})^{\frac{\theta_{2}+\theta_{3}+1}{2}}-s^{\theta_{2}+\theta_{3}+1}\right)ds\notag\\
\sim&\,r^{n+\theta_{1}+\theta_{2}+\theta_{3}}\int^{\frac{1}{\sqrt{2}}}_{0}s^{n+\theta_{1}-2}\left((1-s^{2})^{\frac{\theta_{2}+\theta_{3}+1}{2}}-s^{\theta_{2}+\theta_{3}+1}\right)ds.
\end{align*}
Observe that $\min\{1,2^{-\frac{\theta_{2}+\theta_{3}+1}{2}}\}\leq(1-s^{2})^{\frac{\theta_{2}+\theta_{3}+1}{2}}\leq\max\{1,2^{-\frac{\theta_{2}+\theta_{3}+1}{2}}\}$ in $[0,\frac{1}{\sqrt{2}}].$ Therefore, it is valid iff $\theta_{1}>-(n-1)$ and $n+\theta_{1}+\theta_{2}+\theta_{3}>0$.

We proceed to calculate the last term $I_{3}$. Due to the fact that $|x'|\geq\sqrt{r^{2}-|x'|^{2}}\geq x_{n}$ if $\frac{1}{\sqrt{2}}r\leq|x'|\leq r$ and $0\leq x_{n}\leq\sqrt{r^{2}-|x'|^{2}}$, we obtain
\begin{align*}
I_{3}\sim&\int_{B_{r}'\setminus B'_{\frac{1}{\sqrt{2}}r}}|x'|^{\theta_{1}+\theta_{2}}dx'\int^{\sqrt{r^{2}-|x'|^{2}}}_{0}x_{n}^{\theta_{3}}dx_{n}\notag\\
\sim&\int_{\frac{1}{\sqrt{2}}r}^{r}s^{n+\theta_{1}+\theta_{2}-2}(r^{2}-s^{2})^{\frac{\theta_{3}+1}{2}}ds\sim r^{n+\theta_{1}+\theta_{2}+\theta_{3}}\int_{\frac{1}{\sqrt{2}}}^{1}s^{n+\theta_{1}+\theta_{2}-2}(1-s^{2})^{\frac{\theta_{3}+1}{2}}ds\notag\\
\sim&\,r^{n+\theta_{1}+\theta_{2}+\theta_{3}}.
\end{align*}
This integration is valid iff $\theta_{3}>-1$. A combination of these above facts shows that Lemma \ref{MWQAZ090} holds.

\end{proof}

Based on the measure condition revealed in Lemma \ref{MWQAZ090}, we now make clear its range which lets $w=|x'|^{\theta_{1}}|x|^{\theta_{2}}|x_{n}|^{\theta_{3}}$ become an $A_{p}$-weight as follows.
\begin{theorem}\label{THM000Z}
Set $1<p<\infty$. Then $w=|x'|^{\theta_{1}}|x|^{\theta_{2}}|x_{n}|^{\theta_{3}}$ is an $A_{p}$-weight if $(\theta_{1},\theta_{2},\theta_{3})\in(\mathcal{A}\cup\mathcal{B})\cap(\mathcal{C}_{p}\cup\mathcal{D}_{p})$.
\end{theorem}
\begin{remark}
In fact, the range of $(\theta_{1},\theta_{2},\theta_{3})\in(\mathcal{A}\cup\mathcal{B})\cap(\mathcal{C}_{p}\cup\mathcal{D}_{p})$ captured in Theorem \ref{THM000Z} is optimal. This is due to the fact that the uniform lower bound of $\dashint_{B_{R}(\bar{x})}wdx\big(\dashint_{B_{R}(\bar{x})}w^{-\frac{1}{p-1}}dx\big)^{p-1}\geq C(n,p,\theta_{1},\theta_{2},\theta_{3})$ for any $B_{R}(\bar{x})\subset\mathbb{R}^{n}$ can be established by following the proof of Theorem \ref{Lem009} below with a slight modification.
\end{remark}

\begin{remark}
In the case when $\theta_{1}=\theta_{3}=0$, see pages 505-506 in \cite{G2014} for the corresponding results. With regard to the case of $\theta_{3}=0$, see Theorems 2.3 and 2.6 in \cite{MZ2023}.
\end{remark}

\begin{proof}
For any $\bar{x}=(\bar{x}',\bar{x}_{n})\in\mathbb{R}^{n}$ and $R>0$, we prepare to complete the proof by dividing into the following three cases.

\noindent{\bf Case 1.} Let $|\bar{x}_{n}|\geq3R$. Observe first that
\begin{align}\label{ZQ00}
&\int_{B_{R}(\bar{x})}|x'|^{\theta_{1}}|x|^{\theta_{2}}|x_{n}|^{\theta_{3}}dx\notag\\
&\leq (|\bar{x}|+R\mathrm{sgn}(\theta_{2}))^{\theta_{2}}(|\bar{x}_{n}|+R\mathrm{sgn}(\theta_{3}))^{\theta_{3}}\int_{B_{R}(\bar{x})}|x'|^{\theta_{1}}dx\notag\\
&\leq C(n,\theta_{2},\theta_{3})|\bar{x}|^{\theta_{2}}|\bar{x}_{n}|^{\theta_{3}}R\int_{B_{R}'(\bar{x}')}|x'|^{\theta_{1}}dx',
\end{align}
where $\mathrm{sgn}$ is the sign function as follows:
\begin{align}\label{SGN0}
\mathrm{sgn}(\theta)=&
\begin{cases}
1,&\text{if }\theta>0,\\
0,&\text{if }\theta=0,\\
-1,&\text{if }\theta<0.
\end{cases}
\end{align}
Note that if $|\bar{x}'|\geq3R$, then
\begin{align}\label{ZQ091}
\int_{B'_{R}(\bar{x}')}|x'|^{\theta_{1}}dx'\leq&\omega_{n-1}R^{n-1}
\begin{cases}
\big(|\bar{x}'|+R\big)^{\theta_{1}},&\text{if }\theta_{1}\geq0,\\
\big(|\bar{x}'|-R\big)^{\theta_{1}},&\text{if }\theta_{1}<0
\end{cases}\notag\\
\leq& C(n,\theta_{1}) R^{n-1}|\bar{x}'|^{\theta_{1}},
\end{align}
while, if $|\bar{x}'|<3R$,
\begin{align}\label{ZQ092}
\int_{B'_{R}(\bar{x}')}|x'|^{\theta_{1}}dx'\leq&\int_{B'_{4R}(0')}|x'|^{\theta_{1}}dx'=\frac{(n-1)\omega_{n-1}}{n-1+\theta_{1}}(4R)^{n-1+\theta_{1}}.
\end{align}
Inserting \eqref{ZQ091}--\eqref{ZQ092} into \eqref{ZQ00}, we obtain that if $(\theta_{1},\theta_{2},\theta_{3})\in\mathcal{A}\cup\mathcal{B}$,
\begin{align*}
\dashint_{B_{R}(\bar{x})}|x'|^{\theta_{1}}|x|^{\theta_{2}}|x_{n}|^{\theta_{3}}dx\leq C(n,\theta_{1},\theta_{2},\theta_{3})|\bar{x}|^{\theta_{2}}|\bar{x}_{n}|^{\theta_{3}}
\begin{cases}
|\bar{x}'|^{\theta_{1}},&\text{for }|\bar{x}'|\geq3R,\\
R^{\theta_{1}},&\text{for }|\bar{x}'|<3R.
\end{cases}
\end{align*}
In exactly the same way, we deduce that if $(\theta_{1},\theta_{2},\theta_{3})\in\mathcal{C}_{p}\cup\mathcal{D}_{p}$,
\begin{align*}
&\dashint_{B_{R}(\bar{x})}|x'|^{-\frac{\theta_{1}}{p-1}}|x|^{-\frac{\theta_{2}}{p-1}}|x_{n}|^{-\frac{\theta_{3}}{p-1}}dx\notag\\
&\leq C(n,p,\theta_{1},\theta_{2},\theta_{3})|\bar{x}|^{-\frac{\theta_{2}}{p-1}}|\bar{x}_{n}|^{-\frac{\theta_{3}}{p-1}}
\begin{cases}
|\bar{x}'|^{-\frac{\theta_{1}}{p-1}},&\text{for }|\bar{x}'|\geq3R,\\
R^{-\frac{\theta_{1}}{p-1}},&\text{for }|\bar{x}'|<3R.
\end{cases}
\end{align*}

\noindent{\bf Case 2.} Let $|\bar{x}_{n}|<3R$ and $|\bar{x}|\geq3R$. Then we have
\begin{align}\label{ZW865}
&\int_{B_{R}(\bar{x})}|x'|^{\theta_{1}}|x|^{\theta_{2}}|x_{n}|^{\theta_{3}}dx\notag\\
&\leq (|\bar{x}|+R\mathrm{sgn}(\theta_{2}))^{\theta_{2}}\int_{B_{R}(\bar{x})}|x'|^{\theta_{1}}|x_{n}|^{\theta_{3}}dx\notag\\
&\leq C(n,\theta_{2})|\bar{x}|^{\theta_{2}}\int_{B_{R}'(\bar{x}')}|x'|^{\theta_{1}}dx'\int^{\bar{x}_{n}+\sqrt{R^{2}-|x'-\bar{x}'|^{2}}}_{\bar{x}_{n}-\sqrt{R^{2}-|x'-\bar{x}'|^{2}}}|x_{n}|^{\theta_{3}}dx_{n}.
\end{align}
On one hand, if $|\bar{x}_{n}|\leq\sqrt{R^{2}-|x'-\bar{x}'|^{2}}$, then
\begin{align}\label{ZQ095}
&\int^{\bar{x}_{n}+\sqrt{R^{2}-|x'-\bar{x}'|^{2}}}_{\bar{x}_{n}-\sqrt{R^{2}-|x'-\bar{x}'|^{2}}}|x_{n}|^{\theta_{3}}dx_{n}\notag\\
&=\frac{\big(\bar{x}_{n}+\sqrt{R^{2}-|x'-\bar{x}'|^{2}}\big)^{\theta_{3}+1}+\big(\sqrt{R^{2}-|x'-\bar{x}'|^{2}}-\bar{x}_{n}\big)^{\theta_{3}+1}}{\theta_{3}+1}\notag\\
&\leq\frac{2^{\theta_{3}+2}\big(R^{2}-|x'-\bar{x}'|^{2}\big)^{\frac{\theta_{3}+1}{2}}}{\theta_{3}+1}\leq\frac{2^{\theta_{3}+2}R^{\theta_{3}+1}}{\theta_{3}+1}.
\end{align}
On the other hand, if $|\bar{x}_{n}|>\sqrt{R^{2}-|x'-\bar{x}'|^{2}}$, in light of $|\bar{x}_{n}|<3R$, we have
\begin{align}\label{ZQ096}
&\int^{\bar{x}_{n}+\sqrt{R^{2}-|x'-\bar{x}'|^{2}}}_{\bar{x}_{n}-\sqrt{R^{2}-|x'-\bar{x}'|^{2}}}|x_{n}|^{\theta_{3}}dx_{n}\notag\\
&=\frac{\big(|\bar{x}_{n}|+\sqrt{R^{2}-|x'-\bar{x}'|^{2}}\big)^{\theta_{3}+1}-\big(|\bar{x}_{n}|-\sqrt{R^{2}-|x'-\bar{x}'|^{2}}\big)^{\theta_{3}+1}}{\theta_{3}+1}\notag\\
&\leq\frac{(4R)^{\theta_{3}+1}}{\theta_{3}+1}.
\end{align}
Substituting \eqref{ZQ091}--\eqref{ZQ092} and \eqref{ZQ095}--\eqref{ZQ096} into \eqref{ZW865}, it follows that if $(\theta_{1},\theta_{2},\theta_{3})\in\mathcal{A}\cup\mathcal{B}$,
\begin{align}\label{DA9032}
\dashint_{B_{R}(\bar{x})}|x'|^{\theta_{1}}|x|^{\theta_{2}}|x_{n}|^{\theta_{3}}dx\leq C(n,\theta_{1},\theta_{2},\theta_{3})|\bar{x}|^{\theta_{2}}R^{\theta_{3}}
\begin{cases}
|\bar{x}'|^{\theta_{1}},&\text{for }|\bar{x}'|\geq3R,\\
R^{\theta_{1}},&\text{for }|\bar{x}'|<3R.
\end{cases}
\end{align}
By the same arguments, if $(\theta_{1},\theta_{2},\theta_{3})\in\mathcal{C}_{p}\cup\mathcal{D}_{p}$, then we obtain
\begin{align*}
&\dashint_{B_{R}(\bar{x})}|x'|^{-\frac{\theta_{1}}{p-1}}|x|^{-\frac{\theta_{2}}{p-1}}|x_{n}|^{-\frac{\theta_{3}}{p-1}}dx\notag\\
&\leq C(n,p,\theta_{1},\theta_{2},\theta_{3})|\bar{x}|^{-\frac{\theta_{2}}{p-1}}R^{-\frac{\theta_{3}}{p-1}}
\begin{cases}
|\bar{x}'|^{-\frac{\theta_{1}}{p-1}},&\text{for }|\bar{x}'|\geq3R,\\
R^{-\frac{\theta_{1}}{p-1}},&\text{for }|\bar{x}'|<3R.
\end{cases}
\end{align*}

\noindent{\bf Case 3.} Let $|\bar{x}_{n}|<3R$ and $|\bar{x}|<3R$. It then follows from Lemma \ref{MWQAZ090} that if $(\theta_{1},\theta_{2},\theta_{3})\in\mathcal{A}\cup\mathcal{B}$,
\begin{align*}
\dashint_{B_{R}(\bar{x})}|x'|^{\theta_{1}}|x|^{\theta_{2}}|x_{n}|^{\theta_{3}}dx\leq 4^{n}\dashint_{B_{4R}(0)}|x'|^{\theta_{1}}|x|^{\theta_{2}}|x_{n}|^{\theta_{3}}dx\leq C(n,\theta_{1},\theta_{2},\theta_{3})R^{\theta_{1}+\theta_{2}+\theta_{3}},
\end{align*}
and, if $(\theta_{1},\theta_{2},\theta_{3})\in\mathcal{C}_{p}\cup\mathcal{D}_{p}$,
\begin{align*}
&\dashint_{B_{R}(\bar{x})}|x'|^{-\frac{\theta_{1}}{p-1}}|x|^{-\frac{\theta_{2}}{p-1}}|x_{n}|^{-\frac{\theta_{3}}{p-1}}dx\notag\\
&\leq 4^{n}\dashint_{B_{4R}(0)}|x'|^{-\frac{\theta_{1}}{p-1}}|x|^{-\frac{\theta_{2}}{p-1}}|x_{n}|^{-\frac{\theta_{3}}{p-1}}dx\leq C(n,p,\theta_{1},\theta_{2},\theta_{3})R^{-\frac{\theta_{1}+\theta_{2}+\theta_{3}}{p-1}}.
\end{align*}
The proof is complete.

\end{proof}

In order to understand the Radon measure $d\mu=wdx$ more clearly, we proceed to investigate its doubling property. From Chapters 1 and 15 of \cite{HKM2006}, we see that the Radon measure $d\mu=wdx$ formed by the $A_{p}$-weight $|x'|^{\theta}|x|^{\theta_{2}}|x_{n}|^{\theta_{3}}$ is doubling if $(\theta_{1},\theta_{2},\theta_{3})\in(\mathcal{A}\cup\mathcal{B})\cap(\mathcal{C}_{p}\cup\mathcal{D}_{p})$. In the following we prove that this range is just a sufficient condition but not the necessary condition. The definition of doubling is now stated as follows.
\begin{definition}
A Radon measure $d\mu$ is called doubling provided that there is some constant $0<C<\infty$ such that $\mu(B_{2R}(\bar{x}))\leq C\mu(B_{R}(\bar{x}))$ for any $\bar{x}\in\mathbb{R}^{n}$ and $R>0$.
\end{definition}
\begin{theorem}\label{Lem009}
The Radon measure $d\mu=wdx$ is doubling when $(\theta_{1},\theta_{2},\theta_{3})\in\mathcal{A}\cup\mathcal{B}$.
\end{theorem}

\begin{remark}
The range of $(\theta_{1},\theta_{2},\theta_{3})\in\mathcal{A}\cup\mathcal{B}$ is the necessary and sufficient condition to make $d\mu=wdx$ satisfy the doubling property, since this range is the largest to ensure that $d\mu$ is a Radon measure.
\end{remark}

\begin{proof}
For any given $\bar{x}=(\bar{x}',\bar{x}_{n})\in\mathbb{R}^{n}$ and $R>0$, we divide into three subparts to complete the proof in the following.

\noindent{\bf Step 1.}
Let $|\bar{x}_{n}|\geq 3R$. On one hand, similar to the proof of Case 1 in Theorem \ref{THM000Z}, we have
\begin{align}\label{M098}
&\int_{B_{2R}(\bar{x})}|x'|^{\theta_{1}}|x|^{\theta_{2}}|x_{n}|^{\theta_{3}}dx\notag\\
&\leq C(n,\theta_{1},\theta_{2},\theta_{3})|\bar{x}|^{\theta_{2}}|\bar{x}_{n}|^{\theta_{3}}R^{n}
\begin{cases}
|\bar{x}'|^{\theta_{1}},&\text{for }|\bar{x}'|\geq3R,\\
R^{\theta_{1}},&\text{for }|\bar{x}'|<3R.
\end{cases}
\end{align}

On the other hand,
\begin{align}\label{DA90}
&\int_{B_{R}(\bar{x})}|x'|^{\theta_{1}}|x|^{\theta_{2}}|x_{n}|^{\theta_{3}}dx\notag\\
&\geq (|\bar{x}|-R\mathrm{sgn}(\theta_{2}))^{\theta_{2}}(|\bar{x}_{n}|-R\mathrm{sgn}(\theta_{3}))^{\theta_{3}}\int_{B_{R}(\bar{x})}|x'|^{\theta_{1}}dx\notag\\
&\geq C(n,\theta_{2},\theta_{3})|\bar{x}|^{\theta_{2}}|\bar{x}_{n}|^{\theta_{3}}R\int_{B_{R/2}'(\bar{x}')}|x'|^{\theta_{1}}dx',
\end{align}
where the notation $\mathrm{sgn}$ is defined by \eqref{SGN0} and in the last inequality we utilized the fact that
\begin{align*}
\int_{B_{R}(\bar{x})}|x'|^{\theta_{1}}dx=&2\int_{B_{R}'(\bar{x}')}|x'|^{\theta_{1}}\sqrt{R^{2}-|x'-\bar{x}'|^{2}}dx'\geq\sqrt{3}R\int_{B'_{R/2}(\bar{x}')}|x'|^{\theta_{1}}dx'.
\end{align*}
Since $|x'|^{\theta_{1}}$ increases radially for $\theta_{1}\geq0$ and it decreases radially if $\theta_{1}<0$, it then follows that if $|\bar{x}'|\geq\frac{3}{2}R$,
\begin{align}\label{DQ806}
\int_{B'_{R/2}(\bar{x}')}|x'|^{\theta_{1}}dx'\geq&\omega_{n-1}\Big(\frac{R}{2}\Big)^{n-1}
\big(|\bar{x}'|-\mathrm{sgn}(\theta_{1})R/2\big)^{\theta_{1}}\notag\\
\geq& C(n,\theta_{1})R^{n-1}|\bar{x}'|^{\theta_{1}},
\end{align}
while, if $|\bar{x}'|<\frac{3}{2}R,$
\begin{align}\label{DQ809}
\int_{B'_{R/2}(\bar{x}')}|x'|^{\theta_{1}}dx'\geq&
\begin{cases}
\int_{B'_{R/2}(0')}|x'|^{\theta_{1}}dx',&\text{if }\theta_{1}\geq0,\vspace{0.3ex}\\
\int_{B'_{R/2}(\frac{3}{2}R\frac{\bar{x}'}{|\bar{x}'|})}|x'|^{\theta_{1}}dx',&\text{if }\theta_{1}<0
\end{cases}\notag\\
\geq&\omega_{n-1}R^{n-1+\theta_{1}}
\begin{cases}
\frac{n-1}{2^{n-1+\theta_{1}}(n-1+\theta_{1})},&\text{if }\theta_{1}\geq0,\vspace{0.3ex}\\
\frac{1}{2^{n-1}},&\text{if }\theta_{1}<0.
\end{cases}
\end{align}
Substituting these two relations into \eqref{DA90}, we derive
\begin{align*}
&\int_{B_{R}(\bar{x})}|x'|^{\theta_{1}}|x|^{\theta_{2}}|x_{n}|^{\theta_{3}}dx\geq C(n,\theta_{1},\theta_{2},\theta_{3})|\bar{x}|^{\theta_{2}}|\bar{x}_{n}|^{\theta_{3}}R^{n}
\begin{cases}
|\bar{x}'|^{\theta_{1}},&\text{if }|\bar{x}'|\geq\frac{3}{2}R,\\
R^{\theta_{1}},&\text{if }|\bar{x}'|<\frac{3}{2}R,
\end{cases}
\end{align*}
which, together with \eqref{M098}, reads that Theorem \ref{Lem009} holds in the case of $|\bar{x}_{n}|\geq 3R$.

\noindent{\bf Step 2.}
Let $|\bar{x}_{n}|<3R$ and $|\bar{x}|\geq3R$. To begin with, we have
\begin{align}\label{WM810}
&\int_{B_{R}(\bar{x})}|x'|^{\theta_{1}}|x|^{\theta_{2}}|x_{n}|^{\theta_{3}}dx\notag\\
&\geq (|\bar{x}|-R\mathrm{sgn}(\theta_{2}))^{\theta_{2}}\int_{B_{R}(\bar{x})}|x'|^{\theta_{1}}|x_{n}|^{\theta_{3}}dx\notag\\
&\geq C(n,\theta_{2})|\bar{x}|^{\theta_{2}}\int_{B_{R/2}'(\bar{x}')}|x'|^{\theta_{1}}dx'\int^{\bar{x}_{n}+\sqrt{R^{2}-|x'-\bar{x}'|^{2}}}_{\bar{x}_{n}-\sqrt{R^{2}-|x'-\bar{x}'|^{2}}}|x_{n}|^{\theta_{3}}dx_{n}.
\end{align}
Utilizing \eqref{INE001}, we deduce that for $x'\in B'_{R/2}(\bar{x}'),$

$(i)$ if $|\bar{x}_{n}|\leq\sqrt{R^{2}-|x'-\bar{x}'|^{2}}$,
\begin{align}\label{DQ812}
&\int^{\bar{x}_{n}+\sqrt{R^{2}-|x'-\bar{x}'|^{2}}}_{\bar{x}_{n}-\sqrt{R^{2}-|x'-\bar{x}'|^{2}}}|x_{n}|^{\theta_{3}}dx_{n}\notag\\
&=\frac{\big(\bar{x}_{n}+\sqrt{R^{2}-|x'-\bar{x}'|^{2}}\big)^{\theta_{3}+1}+\big(\sqrt{R^{2}-|x'-\bar{x}'|^{2}}-\bar{x}_{n}\big)^{\theta_{3}+1}}{\theta_{3}+1}\notag\\
&\geq\frac{2^{\theta_{3}+1}\min\{1,2^{-\theta_{3}}\}}{\theta_{3}+1}(R^{2}-|x'-\bar{x}'|^{2})^{\frac{\theta_{3}+1}{2}}\geq\frac{3^{\frac{\theta_{3}+1}{2}}\min\{1,2^{-\theta_{3}}\}R^{\theta_{3}+1}}{\theta_{3}+1};
\end{align}

$(ii)$ if $|\bar{x}_{n}|>\sqrt{R^{2}-|x'-\bar{x}'|^{2}}$,
\begin{align}\label{DQ815}
&\int^{\bar{x}_{n}+\sqrt{R^{2}-|x'-\bar{x}'|^{2}}}_{\bar{x}_{n}-\sqrt{R^{2}-|x'-\bar{x}'|^{2}}}|x_{n}|^{\theta_{3}}dx_{n}\notag\\
&=\frac{\big(|\bar{x}_{n}|+\sqrt{R^{2}-|x'-\bar{x}'|^{2}}\big)^{\theta_{3}+1}-\big(|\bar{x}_{n}|-\sqrt{R^{2}-|x'-\bar{x}'|^{2}}\big)^{\theta_{3}+1}}{\theta_{3}+1}\notag\\
&\geq(R^{2}-|x'-\bar{x}'|^{2})^{\frac{\theta_{3}+1}{2}}
\begin{cases}
\frac{2^{\theta_{3}+1}}{\theta_{3}+1},&\text{for }\theta_{3}\geq0,\\
2((1+\sqrt{3})^{\theta_{3}+1}-3^{\frac{\theta_{3}+1}{2}}),&\text{for }-1<\theta_{3}<0
\end{cases}\notag\\
&\geq R^{\theta_{3}+1}
\begin{cases}
\frac{3^{\frac{\theta_{3}+1}{2}}}{\theta_{3}+1},&\text{for }\theta_{3}\geq0,\\
2((3+\sqrt{3})^{\theta_{3}+1}-3^{\theta_{3}+1}),&\text{for }-1<\theta_{3}<0,
\end{cases}
\end{align}
where we also utilized the following inequality with $t=\frac{|\bar{x}_{n}|-\sqrt{R^{2}-|x'-\bar{x}'|^{2}}}{2\sqrt{R^{2}-|x'-\bar{x}'|^{2}}}$: for $-1<\theta_{3}<0$,
\begin{align*}
(1+t)^{\theta_{3}+1}-t^{\theta_{3}+1}\geq(1+\sqrt{3})^{\theta_{3}+1}-3^{\frac{\theta_{3}+1}{2}},\quad\text{for }0<t\leq\sqrt{3}.
\end{align*}
Therefore, inserting \eqref{DQ806}--\eqref{DQ809} and \eqref{DQ812}--\eqref{DQ815} into \eqref{WM810}, we deduce
\begin{align*}
&\int_{B_{R}(\bar{x})}|x'|^{\theta_{1}}|x|^{\theta_{2}}|x_{n}|^{\theta_{3}}dx\geq C(n,\theta_{1},\theta_{2},\theta_{3})|\bar{x}|^{\theta_{2}}R^{n+\theta_{3}}
\begin{cases}
|\bar{x}'|^{\theta_{1}},&\text{if }|\bar{x}'|\geq\frac{3}{2}R,\\
R^{\theta_{1}},&\text{if }|\bar{x}'|<\frac{3}{2}R.
\end{cases}
\end{align*}
Applying \eqref{DA9032} with $R$ replaced by $2R$, we have
\begin{align*}
\int_{B_{2R}(\bar{x})}|x'|^{\theta_{1}}|x|^{\theta_{2}}|x_{n}|^{\theta_{3}}dx\leq C(n,\theta_{1},\theta_{2},\theta_{3})|\bar{x}|^{\theta_{2}}R^{n+\theta_{3}}
\begin{cases}
|\bar{x}'|^{\theta_{1}},&\text{for }|\bar{x}'|\geq3R,\\
R^{\theta_{1}},&\text{for }|\bar{x}'|<3R.
\end{cases}
\end{align*}
Combining these two above relations, we obtain that Theorem \ref{Lem009} holds if $|\bar{x}_{n}|<3R$ and $|\bar{x}|\geq3R$.

\noindent{\bf Step 3.} Let $|\bar{x}_{n}|<3R$ and $|\bar{x}|<3R$. First, observe from Lemma \ref{MWQAZ090} that
\begin{align*}
\int_{B_{2R}(\bar{x})}|x'|^{\theta_{1}}|x|^{\theta_{2}}|x_{n}|^{\theta_{3}}dx\leq C(n,\theta_{1},\theta_{2},\theta_{3})R^{n+\theta_{1}+\theta_{2}+\theta_{3}}.
\end{align*}
We now proceed to calculate the same order lower bound of $\int_{B_{R}(\bar{x})}|x'|^{\theta_{1}}|x|^{\theta_{2}}|x_{n}|^{\theta_{3}}dx$ with respect to $R$. In the following we divide into five subcases to complete the proof.

{\bf Case 1.} If $\theta_{1}<0$ and $\theta_{3}\leq0$, it follows that
\begin{align*}
\int_{B_{R}(\bar{x})}|x'|^{\theta_{1}}|x|^{\theta_{2}}|x_{n}|^{\theta_{3}}dx\geq&\int_{B_{R}(\bar{x})}|x|^{\theta_{1}+\theta_{2}+\theta_{3}}dx\notag\\
\geq&
\begin{cases}
\int_{B_{R}(0)}|x|^{\theta_{1}+\theta_{2}+\theta_{3}}dx,&\text{if }\theta_{1}+\theta_{2}+\theta_{3}\geq0,\\
\int_{B_{R}(3R\frac{\bar{x}}{|\bar{x}|})}|x|^{\theta_{1}+\theta_{2}+\theta_{3}}dx,&\text{if }\theta_{1}+\theta_{2}+\theta_{3}<0
\end{cases}\notag\\
\geq&R^{n+\theta_{1}+\theta_{2}+\theta_{3}}
\begin{cases}
\frac{n\omega_{n}}{n+\theta_{1}+\theta_{2}+\theta_{3}},&\text{if }\theta_{1}+\theta_{2}+\theta_{3}\geq0,\\
2^{\theta_{1}+\theta_{2}+\theta_{3}}\omega_{n},&\text{if } \theta_{1}+\theta_{2}+\theta_{3}<0.
\end{cases}
\end{align*}

{\bf Case 2.} If $\theta_{1}<0$, $\theta_{3}>0$, and $\theta_{1}+\theta_{2}\geq0$, then
\begin{align*}
\int_{B_{R}(\bar{x})}|x'|^{\theta_{1}}|x|^{\theta_{2}}|x_{n}|^{\theta_{3}}dx\geq&\int_{B_{R}(\bar{x})}|x|^{\theta_{1}+\theta_{2}}|x_{n}|^{\theta_{3}}dx\geq\int_{B_{R}(\bar{x})}|x_{n}|^{\theta_{1}+\theta_{2}+\theta_{3}}dx\notag\\
\geq&\int_{B'_{R/2}(\bar{x}')}dx'\int_{\bar{x}_{n}-\sqrt{R^{2}-|x'-\bar{x}'|^{2}}}^{\bar{x}_{n}+\sqrt{R^{2}-|x'-\bar{x}'|^{2}}}|x_{n}|^{\theta_{1}+\theta_{2}+\theta_{3}}dx_{n}.
\end{align*}
For $x'\in B_{R/2}'(\bar{x}')$, we deduce from \eqref{INE001} that when $|\bar{x}_{n}|\leq \sqrt{R^{2}-|x'-\bar{x}'|^{2}}$,
\begin{align*}
&\int_{\bar{x}_{n}-\sqrt{R^{2}-|x'-\bar{x}'|^{2}}}^{\bar{x}_{n}+\sqrt{R^{2}-|x'-\bar{x}'|^{2}}}|x_{n}|^{\theta_{1}+\theta_{2}+\theta_{3}}dx_{n}\notag\\
&=\frac{(\bar{x}_{n}+\sqrt{R^{2}-|x'-\bar{x}'|^{2}})^{\theta_{1}+\theta_{2}+\theta_{3}+1}+(\sqrt{R^{2}-|x'-\bar{x}'|^{2}}-\bar{x}_{n})^{\theta_{1}+\theta_{2}+\theta_{3}+1}}{\theta_{1}+\theta_{2}+\theta_{3}+1}\notag\\
&\geq\frac{2(R^{2}-|x'-\bar{x}'|^{2})^{\frac{\theta_{1}+\theta_{2}+\theta_{3}+1}{2}}}{\theta_{1}+\theta_{2}+\theta_{3}+1}\geq\frac{3^{\frac{\theta_{1}+\theta_{2}+\theta_{3}+1}{2}}R^{\theta_{1}+\theta_{2}+\theta_{3}+1}}{2^{\theta_{1}+\theta_{2}+\theta_{3}}(\theta_{1}+\theta_{2}+\theta_{3}+1)},
\end{align*}
and, when $|\bar{x}_{n}|>\sqrt{R^{2}-|x'-\bar{x}'|^{2}}$,
\begin{align*}
&\int_{\bar{x}_{n}-\sqrt{R^{2}-|x'-\bar{x}'|^{2}}}^{\bar{x}_{n}+\sqrt{R^{2}-|x'-\bar{x}'|^{2}}}|x_{n}|^{\theta_{1}+\theta_{2}+\theta_{3}}dx_{n}\notag\\
&=\frac{(|\bar{x}_{n}|+\sqrt{R^{2}-|x'-\bar{x}'|^{2}})^{\theta_{1}+\theta_{2}+\theta_{3}+1}-(|\bar{x}_{n}|-\sqrt{R^{2}-|x'-\bar{x}'|^{2}})^{\theta_{1}+\theta_{2}+\theta_{3}+1}}{\theta_{1}+\theta_{2}+\theta_{3}+1}\notag\\
&\geq\frac{2^{\theta_{1}+\theta_{2}+\theta_{3}+1}(R^{2}-|x'-\bar{x}'|^{2})^{\frac{\theta_{1}+\theta_{2}+\theta_{3}+1}{2}}}{\theta_{1}+\theta_{2}+\theta_{3}+1}\geq\frac{3^{\frac{\theta_{1}+\theta_{2}+\theta_{3}+1}{2}}R^{\theta_{1}+\theta_{2}+\theta_{3}+1}}{\theta_{1}+\theta_{2}+\theta_{3}+1}.
\end{align*}
A combination of these above facts shows that
\begin{align*}
\int_{B_{R}(\bar{x})}|x'|^{\theta_{1}}|x|^{\theta_{2}}|x_{n}|^{\theta_{3}}dx\geq C(n,\theta_{1},\theta_{2},\theta_{3})R^{n+\theta_{1}+\theta_{2}+\theta_{3}}.
\end{align*}

{\bf Case 3.} If $\theta_{1}<0$, $\theta_{3}>0$, and $\theta_{1}+\theta_{2}<0$, it follows from \eqref{DQ812}--\eqref{DQ815} that
\begin{align*}
\int_{B_{R}(\bar{x})}|x'|^{\theta_{1}}|x|^{\theta_{2}}|x_{n}|^{\theta_{3}}dx\geq&\int_{B_{R}(\bar{x})}|x|^{\theta_{1}+\theta_{2}}|x_{n}|^{\theta_{3}}dx\geq(4R)^{\theta_{1}+\theta_{2}}\int_{B_{R}(\bar{x})}|x_{n}|^{\theta_{3}}dx\notag\\
\geq&(4R)^{\theta_{1}+\theta_{2}}\int_{B_{R/2}'(\bar{x}')}dx'\int^{\bar{x}_{n}+\sqrt{R^{2}-|x'-\bar{x}'|^{2}}}_{\bar{x}_{n}-\sqrt{R^{2}-|x'-\bar{x}'|^{2}}}|x_{n}|^{\theta_{3}}dx_{n}\notag\\
\geq& C(n,\theta_{1},\theta_{2},\theta_{3})R^{n+\theta_{1}+\theta_{2}+\theta_{3}},
\end{align*}
where we also made use of the fact that $|x|\leq|x-\bar{x}|+|\bar{x}|<4R$.

{\bf Case 4.} If $\theta_{1}\geq0$ and $\theta_{2}\geq0$, applying the proofs of \eqref{DQ806}--\eqref{DQ809} with minor modification and making use of \eqref{DQ812}--\eqref{DQ815} again, we obtain
\begin{align*}
\int_{B_{R}(\bar{x})}|x'|^{\theta_{1}}|x|^{\theta_{2}}|x_{n}|^{\theta_{3}}dx\geq&\int_{B'_{R/2}(\bar{x}')}|x'|^{\theta_{1}+\theta_{2}}dx'\int^{\bar{x}_{n}+\sqrt{R^{2}-|x'-\bar{x}'|^{2}}}_{\bar{x}_{n}-\sqrt{R^{2}-|x'-\bar{x}'|^{2}}}|x_{n}|^{\theta_{3}}dx_{n}\notag\\
\geq& C(n,\theta_{1},\theta_{2},\theta_{3})R^{n+\theta_{1}+\theta_{2}+\theta_{3}}.
\end{align*}

{\bf Case 5.} If $\theta_{1}\geq0$ and $\theta_{2}<0$, it follows from \eqref{DQ806}--\eqref{DQ809} and \eqref{DQ812}--\eqref{DQ815} that
\begin{align*}
\int_{B_{R}(\bar{x})}|x'|^{\theta_{1}}|x|^{\theta_{2}}|x_{n}|^{\theta_{3}}dx\geq&(4R)^{\theta_{2}}\int_{B'_{R/2}(\bar{x}')}|x'|^{\theta_{1}}dx'\int^{\bar{x}_{n}+\sqrt{R^{2}-|x'-\bar{x}'|^{2}}}_{\bar{x}_{n}-\sqrt{R^{2}-|x'-\bar{x}'|^{2}}}|x_{n}|^{\theta_{3}}dx_{n}\notag\\
\geq& C(n,\theta_{1},\theta_{2},\theta_{3})R^{n+\theta_{1}+\theta_{2}+\theta_{3}},
\end{align*}
where we also utilized the fact that $|x|\leq|x-\bar{x}|+|\bar{x}|<4R$. A consequence of these above facts shows that Theorem \ref{Lem009} holds.

\end{proof}

According to the theories of $A_{p}$-weights, we directly obtain the following anisotropic weighted Sobolev and Poincar\'{e} inequalities, which are two fundamental tools to study relevant Sobolev spaces and partial differential equations. Combining Theorems 1.8 and 15.23 in \cite{HKM2006} and Theorem \ref{THM000Z} above, we have the following anisotropic weighted Sobolev embedding theorem.
\begin{corollary}\label{CO901}
Let $n\geq2$, $1<q<p<nq$, $\sum^{3}_{i=1}\theta_{i}\geq n(q-1)$, and $(\theta_{1},\theta_{2},\theta_{3})\in[(\mathcal{A}\cup\mathcal{B})\cap(\mathcal{C}_{p}\cup\mathcal{D}_{p})]\cup\{\theta_{1}=\theta_{3}=0,\,\theta_{2}\geq n(p-1)\}$. Then we obtain that for any $u\in W_{0}^{1,p}(B_{R}(x_{0}),w)$,
\begin{align}\label{QE90}
\|u\|_{L^{p\chi}(B_{R}(x_{0}),w)}\leq C(n,p,q,\theta_{1},\theta_{2},\theta_{3})\|\nabla u\|_{L^{p}(B_{R}(x_{0}),w)},
\end{align}
where $w=|x'|^{\theta_{1}}|x|^{\theta_{2}}|x_{n}|^{\theta_{3}}$ and $\chi$ is given by
\begin{align}\label{CHI}
\chi:=\frac{n+\sum^{3}_{i=1}\theta_{i}}{n+\sum^{3}_{i=1}\theta_{i}-p}.
\end{align}

\end{corollary}
\begin{remark}
The condition of $1<q<p<nq,\,\sum^{3}_{i=1}\theta_{i}\geq n(q-1)$ is assumed to make $\chi\leq\frac{n}{n-p/q}$ for the purpose of meeting the requirement of Theorem 15.23 in \cite{HKM2006}.
\end{remark}
\begin{remark}
When $\theta_{3}=0$, making use of Theorem 1 in \cite{LY2021}, we can enlarge the range of $(\theta_{1},\theta_{2},\theta_{3})$ which lets \eqref{QE90} hold, especially removing the restrictive condition $\mathcal{C}_{q}\cap\mathcal{D}_{q}$. Therefore, it is significant to extend the anisotropic Caffarelli-Kohn-Nirenberg type inequalities in \cite{LY2021} to the case of $\theta_{3}>-1$. More importantly, anisotropic Caffarelli-Kohn-Nirenberg type inequalities can provide a fine parabolic Sobolev inequality with anisotropic weights, which allows us to extend the following regularity results to weighted parabolic equations as shown in \cite{MZ2023,MZ2024}. For more related investigations on Sobolev and Poincar\'{e} inequalities with $A_{p}$-weights covering power-type weights, see \cite{BCG2006,BT2002,CR2013,CRS2016,CW1985,L1986,NS2018,NS2019,DD2023,CW2006} and the references therein.

\end{remark}

Fix $1<p_{0}<\frac{n}{n-1}$ and define
\begin{align*}
\begin{cases}
\mathcal{F}_{p_{0}}=\{(\theta_{1},\theta_{2},\theta_{3}):\,\theta_{1}>-\frac{(n-1)(p_{0}-1)}{p_{0}},\,\theta_{2}\geq0,\,\theta_{3}>-\frac{p_{0}-1}{p_{0}}\},\vspace{0.3ex}\\
\mathcal{G}_{p_{0}}=\{(\theta_{1},\theta_{2},\theta_{3}):\,\theta_{1}>-\frac{(n-1)(p_{0}-1)}{p_{0}},\,\theta_{2}<0,\,\theta_{1}+\theta_{2}+\theta_{3}>-\frac{n(p_{0}-1)}{p_{0}}\}.
\end{cases}
\end{align*}
We now state the anisotropic weighted Poincar\'{e} inequalities as follows.
\begin{lemma}\label{QWZM090}
Let $n\geq2$, $1<p<\infty$, and $1<p_{0}<\frac{n}{n-1}$. Then we have

$(i)$ if $(\theta_{1},\theta_{2},\theta_{3})\in[(\mathcal{A}\cup\mathcal{B})\cap(\mathcal{C}_{p}\cup\mathcal{D}_{p})]\cup\{\theta_{1}=\theta_{3}=0,\,\theta_{2}\geq n(p-1)\}$, then there exists some constant $1<\tilde{p}=\tilde{p}(n,p,\theta_{1},\theta_{2},\theta_{3})<p$ such that for any $B:=B_{R}(\bar{x})\subset\mathbb{R}^{n}$, $R>0$, and $u\in W^{1,\tilde{p}}(B,w)$,
\begin{align}\label{A065}
\int_{B}|u-u_{B}|^{\tilde{p}}d\mu\leq C(n,p,\theta_{1},\theta_{2},\theta_{3})R^{\tilde{p}}\int_{B}|\nabla u|^{\tilde{p}}d\mu;
\end{align}

$(ii)$ if $(\theta_{1},\theta_{2},\theta_{3})\in\mathcal{F}_{p_{0}}\cup\mathcal{G}_{p_{0}}$, then for any $B:=B_{R}\subset\mathbb{R}^{n}$ and $u\in W^{1,1}(B)$,
\begin{align}\label{A095}
\int_{B}|u-u_{B}|d\mu\leq C(n,p_{0},\theta_{1},\theta_{2},\theta_{3})R^{1+\theta_{1}+\theta_{2}+\theta_{3}}\int_{B}|\nabla u|dx,
\end{align}
where $d\mu:=wdx=|x'|^{\theta_{1}}|x|^{\theta_{2}}|x_{n}|^{\theta_{3}}dx$ and $u_{B}=\frac{1}{\mu(B)}\int_{B}u d\mu.$
\end{lemma}
\begin{remark}
It is worth pointing out that we use the theories of quasiconformal mappings rather than that of $A_{p}$-weights to keep \eqref{A065} holding for the weight $|x|^{\theta_{2}}$ if $\theta_{2}\geq n(p-1)$, see Corollary 15.35 in \cite{HKM2006}. This fact also indicates that $A_{p}$-weight is not a necessary condition to ensure that \eqref{A065} holds. Moreover, in contrast to \eqref{A065}, the result in \eqref{A095} holds without restriction in terms of the upper bound on the ranges of $\theta_{i}$, $i=1,2,3$. Then it is natural to ask if \eqref{A065} still holds when we get rid of restrictive condition $\mathcal{C}_{p}\cup \mathcal{D}_{p}$.

\end{remark}

\begin{proof}
To begin with, a combination of Theorem 15.13, Theorem 15.21 and Corollary 15.35 in \cite{HKM2006} and Theorem \ref{THM000Z} above shows that \eqref{A065} holds.

By scaling, it suffices to demonstrate that \eqref{A095} holds when $R=1$. Making use of Lemma 1.4 of Fabes-Kenig-Serapioni \cite{FKS1982}, we have
\begin{align*}
\left|u(x)-\dashint_{B_{1}}udy\right|\leq\dashint_{B_{1}}|u(x)-u(y)|dy\leq C\int_{B_{1}}\frac{|\nabla u(y)|}{|x-y|^{n-1}}dy,\quad\dashint_{B_{1}}:=\frac{1}{|B_{1}|}\int_{B_{1}},
\end{align*}
which reads that
\begin{align}\label{A091}
\int_{B_{1}}\left|u(x)-\dashint_{B_{1}}udy\right|wdx\leq C\int_{B_{1}}\int_{B_{1}}\frac{|x'|^{\theta_{1}}|x|^{\theta_{2}}|x_{n}|^{\theta_{3}}}{|x-y|^{n-1}}|\nabla u(y)| dxdy.
\end{align}
In light of $(\theta_{1},\theta_{2},\theta_{3})\in\mathcal{F}_{p_{0}}\cup\mathcal{G}_{p_{0}}$, we obtain from H\"{o}lder's inequality and Lemma \ref{MWQAZ090} that for $1<p_{0}<\frac{n}{n-1}$,
\begin{align*}
&\int_{B_{1}}\frac{|x'|^{\theta_{1}}|x|^{\theta_{2}}|x_{n}|^{\theta_{3}}}{|x-y|^{n-1}}dx\notag\\
&\leq\bigg(\int_{B_{2}(y)}\frac{1}{|x-y|^{p_{0}(n-1)}}dx\bigg)^{\frac{1}{p_{0}}}\left(\int_{B_{1}}|x'|^{\frac{\theta_{1}p_{0}}{p_{0}-1}}|x|^{\frac{\theta_{2}p_{0}}{p_{0}-1}}|x_{n}|^{\frac{\theta_{3}p_{0}}{p_{0}-1}}dx\right)^{\frac{p_{0}-1}{p_{0}}}\leq C.
\end{align*}
Substituting this into \eqref{A091}, we have
\begin{align*}
\int_{B_{1}}\left|u(x)-\dashint_{B_{1}}udy\right|wdx\leq C\int_{B_{1}}|\nabla u(x)|dx.
\end{align*}
It then follows that
\begin{align*}
\int_{B_{1}}|u(x)-u_{B_{1}}|wdx\leq&\int_{B_{1}}\left|u(x)-\dashint_{B_{1}}u\right|wdx+\int_{B_{1}}\left|\dashint_{B_{1}}u-u_{B_{1}}\right|wdx\notag\\
\leq&2\int_{B_{1}}\left|u(x)-\dashint_{B_{1}}u\right|wdx\leq C\int_{B_{1}}|\nabla u|dx,
\end{align*}
where $u_{B_{1}}=\frac{1}{\mu(B_{1})}\int_{B_{1}}ud\mu$. Then \eqref{A095} holds.

\end{proof}

\section{Applications to anisotropic weighted $p$-Laplace equations}\label{SEC03}

$p$-Laplace equation comes up in a large variety of physical contexts and involves wide applications from game theory to nonlinear flow in porous media, mechanics and image processing etc, see e.g. \cite{DLK2013,K1996,L2019}. These important applications stimulate us to investigate the local behavior of solutions to anisotropically weighted $p$-Laplace equations with non-homogeneous terms in the following.

We first introduce some notations for later use. For $\bar{x}\in\mathbb{R}^{n}$ and $R>0$, define $B^{+}_{R}(\bar{x})=B_{R}(\bar{x})\cap\{(x',x_{n}):x_{n}>0\}$. For brevity, we simplify the notation $B^{+}_{R}(\bar{x})$ as $B^{+}_{R}$ if $\bar{x}=0$. Denote $\partial'B^{+}_{R}(\bar{x})=B_{R}(\bar{x})\cap\{x_{n}=0\}$ and $\partial''B_{R}^{+}(\bar{x})=\partial B_{R}^{+}(\bar{x})\setminus\partial'B^{+}_{R}(\bar{x})$. Given a weight $w$, we denote by $L^{p}(B^{+}_{R}(\bar{x}),w)$ and $W^{1,p}(B^{+}_{R}(\bar{x}),w)$ the weighted $L^{p}$ spaces and weighted Sobolev spaces with their norms, respectively, represented as
\begin{align*}
\begin{cases}
\|u\|_{L^{p}(B^{+}_{R}(\bar{x}),w)}=\Big(\int_{B^{+}_{R}(\bar{x})}|u|^{p}wdx\Big)^{\frac{1}{p}},\\
\|u\|_{W^{1,p}(B^{+}_{R}(\bar{x}),w)}=\Big(\int_{B^{+}_{R}(\bar{x})}|u|^{p}wdx\Big)^{\frac{1}{p}}+\left(\int_{B^{+}_{R}(\bar{x})}|\nabla u|^{p}wdx\right)^{\frac{1}{p}}.
\end{cases}
\end{align*}
For $\phi_{0}\in L^{\infty}(\partial'' B^{+}_{1})$ and $p>1$, we consider the general non-homogeneous weighted $p$-Laplace equations satisfying the Dirichlet condition as follows:
\begin{align}\label{PROBLEM001}
\begin{cases}
-\mathrm{div}(w|\nabla u|^{p-2}\nabla u)=f_{0}+wf_{1},& \mathrm{in}\;B_{1}^{+},\\
u=0,&\text{on }\partial'B^{+}_{1},\\
u=\phi_{0},&\text{on }\partial'' B^{+}_{1}.
\end{cases}
\end{align}
Here $w=|x'|^{\theta_{1}}|x|^{\theta_{2}}|x_{n}|^{\theta_{3}}$, the ranges of $\theta_{i}$, $i=1,2,3$ are prescribed in the following theorems, $f_{0}\in L^{\frac{mp\chi}{m(\chi-1)+\chi(p-1)}}(B_{1}^{+},w^{-\frac{\chi(m-1)(p-1)+m}{m(\chi-1)+\chi(p-1)}})$, and $f_{1}\in L^{\frac{mp\chi}{m(\chi-1)+\chi(p-1)}}(B^{+}_{1},w)$, where $\chi$ is given by \eqref{CHI} and $m>\frac{n+\sum^{3}_{i=1}\theta_{i}}{p}$.
We say that $u\in W^{1,p}(B_{1}^{+},w)$ is termed a weak solution of \eqref{PROBLEM001} provided that for any $\varphi\in W_{0}^{1,p}(B_{1}^{+},w)$,
\begin{align*}
\int_{B_{1}^{+}}w|\nabla u|^{p-2}\nabla u\cdot\nabla\varphi dx=\int_{B_{1}^{+}}(f_{0}+wf_{1})\varphi.
\end{align*}

With regard to the local behavior of solution to problem \eqref{PROBLEM001} near the origin, we have
\begin{theorem}\label{ZWTHM90}
Assume that $n\geq2$, $1<q<p<nq$, $\sum^{3}_{i=1}\theta_{i}\geq n(q-1)$, $(\theta_{1},\theta_{2},\theta_{3})\in[(\mathcal{A}\cup\mathcal{B})\cap(\mathcal{C}_{p}\cup\mathcal{D}_{p})]\cup\{\theta_{1}=\theta_{3}=0,\,\theta_{2}\geq n(p-1)\}$, and $m>\frac{n+\sum^{3}_{i=1}\theta_{i}}{p}$. Let $u$ be the weak solution of problem \eqref{PROBLEM001}. Then there exists a small constant $0<\alpha=\alpha(n,m,p,q,\theta_{1},\theta_{2},\theta_{3})<1$ such that for any $\gamma>0,$
\begin{align}\label{QNAW001ZW}
u(x)=u(0)+O(1)(\|u\|_{L^{\gamma}(B_{1}^{+},w)}+\mathcal{H})|x|^{\alpha},\quad\text{for any }x\in B^{+}_{1/2},
\end{align}
where $O(1)$ satisfies that $|O(1)|\leq C=C(n,m,p,q,\gamma,\theta_{1},\theta_{2},\theta_{3}),$ and
\begin{align}\label{F09}
\mathcal{H}:=&\|f_{0}\|^{\frac{1}{p-1}}_{L^{\frac{mp\chi}{m(\chi-1)+\chi(p-1)}}(B_{1}^{+},w^{-\frac{\chi(m-1)(p-1)+m}{m(\chi-1)+\chi(p-1)}})}+\|f_{1}\|^{\frac{1}{p-1}}_{L^{\frac{mp\chi}{m(\chi-1)+\chi(p-1)}}(B_{1}^{+},w)}.
\end{align}
\end{theorem}

When the weight $|x'|^{\theta_{1}}|x|^{\theta_{2}}|x_{n}|^{\theta_{3}}$ becomes a single power-type weight, that is, only one of $\theta_{i}$, $i=1,2,3$ is nonzero, we establish the following H\"{o}lder estimates.
\begin{theorem}\label{THM002}
Let $n\geq2$, $1<q<p<nq$, $\sum^{3}_{i=1}\theta_{i}\geq n(q-1)$,  $(\theta_{1},\theta_{2},\theta_{3})\in[(\mathcal{A}\cup\mathcal{B})\cap(\mathcal{C}_{p}\cup\mathcal{D}_{p})]\cup\{\theta_{1}=\theta_{3}=0,\,\theta_{2}\geq n(p-1)\}$, and $m>\frac{n+\sum^{3}_{i=1}\theta_{i}}{p}$. Assume that $u$ is the weak solution of problem \eqref{PROBLEM001}. Then there exists some constant $0<\tilde{\alpha}<\frac{\alpha}{1+\alpha}$ with $\alpha$ determined by Theorem \ref{ZWTHM90}, such that when only one of $\theta_{i}$, $i=1,2,3$ is nonzero,
\begin{align}\label{Z010}
|u(x)-u(y)|\leq C(\|\phi_{0}\|_{L^{\infty}(\partial''B^{+}_{1})}+\mathcal{H})|x-y|^{\tilde{\alpha}},\quad\text{for all}\;x,y\in B^{+}_{1/8},
\end{align}
where $\mathcal{H}$ is defined by \eqref{F09}, the positive constants $\tilde{\alpha}$ and $C$ depend only on $n,m,p,q$ and $\max\limits_{1\leq i\leq3}\theta_{i}$.

\end{theorem}
\begin{remark}
First, our results in Theorems \ref{ZWTHM90} and \ref{THM002} can be easily extended to general degenerate elliptic equations with anisotropic weights. Second, when we let $B^{+}_{R_{0}}$ substitute for $B^{+}_{1}$, it follows from the proofs with a slight modification that \eqref{QNAW001ZW} and \eqref{Z010} hold with $B^{+}_{1/2},B^{+}_{1/4}$ replaced by $B^{+}_{R_{0}/2}$ and $B^{+}_{R_{0}/4}$, respectively. Furthermore, the constant $C$ will be dependent on $R_{0}$, but the exponents $\alpha$ and $\tilde{\alpha}$ depend not on $R_{0}$. Additionally, it is worth emphasizing that Jin and Xiong \cite{JX2023} recently studied the H\"{o}lder regularity for the general non-homogeneous linearized fast diffusion equations and porous medium equations with the weight $x_{n}$, which also motivates our study on the local behavior for non-homogeneous $p$-Laplace equations with more general anisotropic weights in this section.

\end{remark}

In the following, we will adopt the classical De Giorgi trunction method created in \cite{D1957} to study the boundedness and local regularity for solution to problem \eqref{PROBLEM001}.

\subsection{Boundedness of weak solutions}
For simplicity, in the following we still denote
$$d\mu=wdx=|x'|^{\theta_{1}}|x|^{\theta_{2}}|x_{n}|^{\theta_{3}}dx.$$
For $E\subset B^{+}_{1}$, let $|E|_{\mu}=\int_{E}wdx$. For $k\in\mathbb{R}$ and $u\in W^{1,p}(B_{1}^{+},w)$, define
\begin{align*}
(u-k)_{+}=\max\{u-k,0\},\quad(u-k)_{-}=\max\{k-u,0\}.
\end{align*}
We start by establishing the Caccioppoli inequality for the truncated solution as follows.
\begin{lemma}\label{lem003}
Assume that $u$ is the solution of \eqref{PROBLEM001}. Let $B_{R}^{+}(x_{0})\subset B^{+}_{1}$. For any $\xi\in C^{\infty}(B_{1})$ satisfying that $0\leq\xi\leq1$ in $B_{1}$ and $\xi=0$ on $\overline{B_{1}^{+}}\setminus (B_{R}^{+}(x_{0})\cup\partial'B_{1}^{+})$, we derive
\begin{align*}
&\int_{B^{+}_{R}(x_{0})}|\nabla(v\xi)|^{p}wdx\leq C\bigg(\int_{B^{+}_{R}(x_{0})}|v\nabla\xi|^{p}w+\mathcal{H}^{p}|\{v>0\}\cap B^{+}_{R}(x_{0})|_{\mu}^{1-\frac{1}{m}}\bigg),
\end{align*}
where $v=(u-k)_{\pm}$, $k\in\mathbb{R}$, and $\mathcal{H}$ is given by \eqref{F09}.

\end{lemma}
\begin{remark}\label{RE06}
When $B_{R}^{+}(x_{0})$ is chosen to be $B_{1}^{+}$ and $v=(\pm u-k)^{+}$ with $k\geq\|\phi_{0}\|_{L^{\infty}(\partial'' B^{+}_{1})}$, Lemma \ref{lem003} holds for any $\xi\in C^{\infty}(B_{1})$ which may not vanish on $\partial'' B^{+}_{1}$, since $v=0$ on $\partial'' B^{+}_{1}$ in this case.
\end{remark}

\begin{proof}
By picking test function $\varphi=v\xi^{p}$, we have
\begin{align*}
\int_{B^{+}_{R}(x_{0})}w|\nabla u|^{p-2}\nabla u\cdot\nabla(v\xi^{p}) dx=\int_{B^{+}_{R}(x_{0})}(f_{0}+wf_{1})v\xi^{p}dx.
\end{align*}
Observe first from H\"{o}lder's inequality and Young's inequality that
\begin{align}\label{KM01}
&\int_{B^{+}_{R}(x_{0})}|\nabla u|^{p-2}\nabla u\nabla(v\xi^{p})w=\int_{B^{+}_{R}(x_{0})}|\nabla v|^{p-2}\nabla v\nabla(v\xi^{p})w\notag\\
&\geq\int_{B^{+}_{R}(x_{0})}|\nabla v|^{p}\xi^{p}w-\int_{B^{+}_{R}(x_{0})}pv\xi^{p-1}|\nabla\xi||\nabla v|^{p-1}w\notag\\
&\geq\int_{B^{+}_{R}(x_{0})}|\nabla v|^{p}\xi^{p}w-p\bigg(\int_{B^{+}_{R}(x_{0})}|\xi\nabla v|^{p}w\bigg)^{\frac{p-1}{p}}\bigg(\int_{B^{+}_{R}(x_{0})}|v\nabla\xi|^{p}w\bigg)^{\frac{1}{p}}\notag\\
&\geq\frac{1}{2}\int_{B^{+}_{R}(x_{0})}|\nabla v|^{p}\xi^{p}w-C(p)\int_{B^{+}_{R}(x_{0})}|v\nabla\xi|^{p}w.
\end{align}

Perform zero extension for $v$ in $B_{1}\setminus B_{1}^{+}$ and still write the extended function as $v$. For simplicity, denote $A(v,R):=B^{+}_{R}(x_{0})\cap\{v>0\}$. A consequence of H\"{o}lder's inequality, Young's inequality and Corollary \ref{CO901} leads to that
\begin{align}\label{KM02}
&\int_{B^{+}_{R}(x_{0})}(f_{0}+wf_{1})v\xi^{p}dx=\int_{A(v,R)}(f_{0}+wf_{1})v\xi^{p}dx\notag\\
&\leq\|v\xi\|_{L^{p\chi}(B_{R}(x_{0}),w)}\Big(\|f_{0}\|_{L^{\frac{p\chi}{p\chi-1}}(A(v,R),w^{-\frac{1}{p\chi-1}})}+\|f_{1}\|_{L^{\frac{p\chi}{p\chi-1}}(A(v,R),w)}\Big)\notag\\
&\leq C\|\nabla(v\xi)\|_{L^{p}(B_{R}(x_{0}),w)}\Big(\|f_{0}\|_{L^{\frac{p\chi}{p\chi-1}}(A(v,R),w^{-\frac{1}{p\chi-1}})}+\|f_{1}\|_{L^{\frac{p\chi}{p\chi-1}}(A(v,R),w)}\Big)\notag\\
&\leq\frac{1}{3}\|\nabla(v\xi)\|_{L^{p}(B^{+}_{R}(x_{0}),w)}^{p}+C\Big(\|f_{0}\|^{\frac{p}{p-1}}_{L^{\frac{p\chi}{p\chi-1}}(A(v,R),w^{-\frac{1}{p\chi-1}})}+\|f_{1}\|^{\frac{p}{p-1}}_{L^{\frac{p\chi}{p\chi-1}}(A(v,R),w)}\Big)\notag\\
&\leq \frac{1}{3}\|\nabla(v\xi)\|_{L^{p}(B^{+}_{R}(x_{0}),w)}^{p}+C\mathcal{H}^{p}|A(v,R)|_{\mu}^{1-\frac{1}{m}},
\end{align}
where $\chi$ and $\mathcal{H}$ are defined by \eqref{CHI} and \eqref{F09}, respectively. Combining \eqref{KM01}--\eqref{KM02} and the fact that $|\nabla(v\xi)|^{p}\leq2^{p-1}(|\xi\nabla v|^{p}+|v\nabla\xi|^{p})$, we obtain that Lemma \ref{lem003} holds.

\end{proof}

Utilizing Lemma \ref{lem003} and Remark \ref{RE06}, we can obtain the following global $L^{\infty}$ estimate.
\begin{theorem}\label{CORO06}
Assume as above. Let $u$ be the solution of \eqref{PROBLEM001}. Then
\begin{align*}
\|u\|_{L^{\infty}(B_{1}^{+})}\leq\|\phi_{0}\|_{L^{\infty}(\partial'' B^{+}_{1})}+C\mathcal{H},
\end{align*}
where $C=C(n,p,q,m,\theta_{1},\theta_{2},\theta_{3})$ and $\mathcal{H}$ is defined by \eqref{F09}.
\end{theorem}
\begin{proof}
Take $k_{i}=\|\phi_{0}\|_{L^{\infty}(\partial'' B^{+}_{1})}+\mathcal{M}-\frac{\mathcal{M}}{2^{i}}$, $i\geq0$ and $\xi\equiv1$ in Lemma \ref{lem003}, where $\mathcal{M}$ is determined in the following. Extend $u$ to be identically zero in $B_{1}\setminus B_{1}^{+}$ and still denote the extended function as $u$. A direct application of Corollary \ref{CO901} and Remark \ref{RE06} reads that
\begin{align*}
&(k_{i+1}-k_{i})^{p}|\{u>k_{i+1}\}\cap B^{+}_{1}|_{\mu}^{\frac{1}{\chi}}\leq\bigg(\int_{B_{1}}|(u-k_{i})_{+}|^{p\chi}wdx\bigg)^{\frac{1}{\chi}}\notag\\
&\leq C\int_{B_{1}}|\nabla(u-k_{i})_{+}|^{p}wdx\leq C\mathcal{H}^{p}|\{u>k_{i}\}\cap B^{+}_{1}|_{\mu}^{1-\frac{1}{m}},
\end{align*}
where $\chi$ and $\mathcal{H}$ are, respectively, given by \eqref{CHI} and \eqref{F09}. Denote
$F_{i}:=|\{u>k_{i}\}\cap B^{+}_{1}|_{\mu}.$
Then we have
\begin{align*}
F_{i+1}\leq&\left(\frac{C\mathcal{H}^{p}2^{(i+1)p}}{\mathcal{M}^{p}}\right)^{\chi}F_{i}^{\frac{\chi(m-1)}{m}}\notag\\
=&\prod_{s=0}^{i}\left[\left(\frac{C\mathcal{H}^{p}2^{(i+1-s)p}}{\mathcal{M}^{p}}\right)^{\frac{m}{m-1}}\right]^{\big(\frac{\chi(m-1)}{m}\big)^{s+1}}F_{0}^{\big(\frac{\chi(m-1)}{m}\big)^{i+1}}\notag\\
\leq&\left[\left(\frac{C_{0}^{p}\mathcal{H}^{p}}{\mathcal{M}^{p}}\right)^{\frac{m\chi}{\chi(m-1)-m}}F_{0}\right]^{\big(\frac{\chi(m-1)}{m}\big)^{i+1}},
\end{align*}
where $\frac{\chi(m-1)}{m}>1$. Picking $\mathcal{M}=C_{0}\mathcal{H}(2|B^{+}_{1}|_{\mu})^{\frac{\chi(m-1)-m}{mp\chi}}$, we deduce
\begin{align*}
F_{i+1}\leq2^{-\big(\frac{\chi(m-1)}{m}\big)^{i+1}}\rightarrow0,\quad\text{as }i\rightarrow\infty,
\end{align*}
which implies that
\begin{align*}
\sup_{B^{+}_{1}}u\leq\|\phi_{0}\|_{L^{\infty}(\partial'' B^{+}_{1})}+C\mathcal{H}.
\end{align*}
Applying the above arguments to the equation of $-u$, we complete the proof.
\end{proof}

We now proceed to give the local $L^{\infty}$ estimate of the solution as follows.
\begin{theorem}\label{THE621}
Assume as above. Suppose that $u$ is the solution of \eqref{PROBLEM001}. Then we obtain that for any $\gamma>0$,
\begin{align*}
\|u\|_{L^{\infty}(B^{+}_{3/4})}\leq C(\|u\|_{L^{\gamma}(B^{+}_{1},w)}+\mathcal{H}),
\end{align*}
where $\mathcal{H}$ is given by \eqref{F09} and $C=C(n,p,q,m,\gamma,\theta_{1},\theta_{2},\theta_{3})$.
\end{theorem}

\begin{proof}
For any given $\tau\in(0,1)$, set
\begin{align*}
r_{i}=\tau+2^{-i}(1-\tau),\quad k_{i}=k(2-2^{-i}),\quad i\geq0,
\end{align*}
where $k>0$ is determined in the following. Pick a nonnegative cut-off function $\xi_{i}\in C^{\infty}(\mathbb{R}^{n})$ satisfying that $\xi_{i}=1$ in $B^{+}_{r_{i+1}}$, $\xi=0$ in $\mathbb{R}^{n}\setminus B_{r_{i}}$, and $|\nabla\xi_{i}|\leq C(n)(r_{i}-r_{i+1})^{-1}$. Denote $A(k,r):=\{u>k\}\cap B^{+}_{r}$ and $|A(k,r)|_{\mu}:=\int_{A(k,r)}wdx$ for $k,r>0$. Extend $u$ to be identically zero in $B_{1}\setminus B_{1}^{+}$ and still write the extended function as $u$. It follows from H\"{o}lder's inequality, Corollary \ref{CO901} and Lemma \ref{lem003} that
\begin{align*}
&\|(u-k_{i+1})_{+}\|_{L^{p}(B^{+}_{r_{i+1}},w)}^{p}\notag\\
&\leq\|\xi_{i}(u-k_{i+1})_{+}\|_{L^{p\chi}(B_{r_{i}},w)}^{p}|A(k_{i+1},r_{i})|_{\mu}^{1-\frac{1}{\chi}}\notag\\
&\leq C\|\nabla(\xi_{i}(u-k_{i+1})_{+})\|_{L^{p}(B_{r_{i}},w)}^{p}|A(k_{i+1},r_{i})|_{\mu}^{1-\frac{1}{\chi}}\notag\\
&\leq C\left(\frac{2^{p(i+1)}}{(1-\tau)^{p}}\|(u-k_{i+1})_{+}\|_{L^{p}(B^{+}_{r_{i}},w)}^{p}+\mathcal{H}^{p}|A(k_{i+1},r_{i})|_{\mu}^{1-\frac{1}{m}}\right)|A(k_{i+1},r_{i})|_{\mu}^{1-\frac{1}{\chi}},
\end{align*}
where $\chi$ is given by \eqref{CHI}. Note that
\begin{align*}
\|(u-k_{i})_{+}\|^{p}_{L^{p}(B^{+}_{r_{i}},w)}\geq&(k_{i+1}-k_{i})^{p}|A(k_{i+1},r_{i})|_{\mu}.
\end{align*}
For $i\geq0$, denote $F_{i}:=k^{-p}\|(u-k_{i})_{+}\|_{L^{p}(B^{+}_{r_{i}},w)}^{p}$ and choose $k\geq\|u\|_{L^{p}(B^{+}_{1},w)}+\mathcal{H}$. It then follows that
\begin{align*}
F_{i+1}\leq&\frac{C2^{p(i+1)(2-\frac{1}{\chi})}}{(1-\tau)^{p}}F_{i}^{2-\frac{1}{\chi}}+C\mathcal{H}^{p}2^{(i+1)p(2-\frac{1}{\chi}-\frac{1}{m})}k^{-p}F_{i}^{2-\frac{1}{\chi}-\frac{1}{m}}\notag\\
\leq&\prod\limits_{s=0}^{i}\left(\frac{C2^{p(i+1-s)(2-\frac{1}{\chi})}}{(1-\tau)^{p}}\right)^{(2-\frac{1}{\chi}-\frac{1}{m})^{s}}F_{0}^{(2-\frac{1}{\chi}-\frac{1}{m})^{i+1}}\notag\\
\leq&\left(\frac{\overline{C}_{0}^{p}F_{0}}{(1-\tau)^{^{p(1-\frac{1}{\chi}-\frac{1}{m})}}}\right)^{(2-\frac{1}{\chi}-\frac{1}{m})^{i+1}},\quad\overline{C}_{0}=\overline{C}_{0}(n,p,m,\theta_{1},\theta_{2},\theta_{3}).
\end{align*}
By choosing $k=\frac{2\overline{C}_{0}}{(1-\tau)^{1-\frac{1}{\chi}-\frac{1}{m}}}\|u\|_{L^{p}(B^{+}_{1},w)}+\mathcal{H}$, we have
\begin{align*}
F_{i+1}\leq2^{-p(2-\frac{1}{\chi}-\frac{1}{m})^{i+1}}\rightarrow0,\quad\text{as } i\rightarrow\infty,
\end{align*}
which implies that
\begin{align*}
\sup_{B^{+}_{\tau}}u\leq2k\leq\frac{4\overline{C}_{0}}{(1-\tau)^{1-\frac{1}{\chi}-\frac{1}{m}}}\|u\|_{L^{p}(B^{+}_{1},w)}+2\mathcal{H}.
\end{align*}
Repeating the above arguments with $u$ replaced by $-u$, we deduce that for any $\tau\in(0,1),$
\begin{align}\label{WE806}
\|u\|_{L^{\infty}(B^{+}_{\tau})}\leq\frac{4\overline{C}_{0}}{(1-\tau)^{1-\frac{1}{\chi}-\frac{1}{m}}}\|u\|_{L^{p}(B^{+}_{1},w)}+2\mathcal{H}.
\end{align}

For $R\in(0,1]$, let $\tilde{u}(x)=u(Rx)$, $\tilde{f}_{0}(x)=f_{0}(Rx)$ and $\tilde{f}_{1}(x)=f_{1}(Rx)$. Then we have
\begin{align*}
-\mathrm{div}(w|\nabla\tilde{u}|^{p-2}\nabla\tilde{u})=R^{p-\sum^{3}_{i=1}\theta_{i}}\tilde{f}_{0}+R^{p}w\tilde{f}_{1},\quad\mathrm{in}\;B^{+}_{1}.
\end{align*}
Then applying \eqref{WE806} to $\tilde{u}$ and rescaling back to $u$, we deduce
\begin{align*}
\|u\|_{L^{\infty}(B^{+}_{\tau R})}\leq&\frac{4\overline{C}_{0}}{(1-\tau)^{1-\frac{1}{\chi}-\frac{1}{m}}R^{\frac{n+\sum^{3}_{i=1}\theta_{i}}{p}}}\|u\|_{L^{p}(B^{+}_{R},w)}+2R^{\frac{m_{1}p-n-\sum^{3}_{i=1}\theta_{i}}{m_{1}(p-1)}}\mathcal{H}\notag\\
\leq&\frac{4\overline{C}_{0}}{(1-\tau)^{1-\frac{1}{\chi}-\frac{1}{m}}R^{\frac{n+\sum^{3}_{i=1}\theta_{i}}{p}}}\|u\|_{L^{p}(B^{+}_{R},w)}+2\mathcal{H},
\end{align*}
where we also utilized the fact that $m>\frac{n+\sum^{3}_{i=1}\theta_{i}}{p}$. For any $\gamma>0$, it follows from H\"{o}lder's inequality that if $\gamma\geq p$,
\begin{align*}
\|u\|_{L^{p}(B^{+}_{R},w)}\leq
\begin{cases}
\|u\|_{L^{\gamma}(B^{+}_{R},w)},\quad\text{for }\gamma=p,\\
C\|u\|_{L^{\gamma}(B^{+}_{R},w_{2})},\quad\text{for }\gamma>p,
\end{cases}
\end{align*}
while, if $0<\gamma<p$, by choosing $\bar{\gamma}=\frac{\gamma}{p}$, we have from Young's inequality that
\begin{align*}
\|u\|_{L^{p}(B^{+}_{R},w)}\leq&\|u\|_{L^{\infty}(B^{+}_{R})}^{\frac{p-\bar{\gamma}}{p}}\left(\int_{B^{+}_{R}}|u|^{\bar{\gamma}}w\right)^{\frac{1}{p}}\leq C\|u\|_{L^{\infty}(B^{+}_{R})}^{\frac{p-\bar{\gamma}}{p}}\|u\|_{L^{\gamma}(B^{+}_{R},w)}^{\frac{\bar{\gamma}}{p}}\notag\\
\leq&\frac{(1-\tau)^{1-\frac{1}{\chi}-\frac{1}{m}}R^{\frac{n+\sum^{3}_{i=1}\theta_{i}}{p}}}{8\overline{C}_{0}}\|u\|_{L^{\infty}(B^{+}_{R})}\notag\\
&+\frac{C}{(1-\tau)^{\frac{(1-\frac{1}{\chi}-\frac{1}{m})(p-\bar{\gamma})}{\bar{\gamma}}}R^{\frac{(n+\sum^{3}_{i=1}\theta_{i})(p-\bar{\gamma})}{p\bar{\gamma}}}}\|u\|_{L^{\gamma}(B^{+}_{R},w)}.
\end{align*}
Combining these above facts, we obtain that for any $\gamma>0,$
\begin{align*}
\|u\|_{L^{\infty}(B^{+}_{\tau R})}\leq\frac{1}{2}\|u\|_{L^{\infty}(B^{+}_{R})}+\frac{C}{(R-\tau R)^{\beta}}\|u\|_{L^{\gamma}(B^{+}_{R},w)}+2\mathcal{H},
\end{align*}
where
\begin{align*}
\beta:=\frac{1}{\bar{\gamma}}\max\bigg\{p\Big(1-\frac{1}{\chi}-\frac{1}{m}\Big),n+\sum^{3}_{i=1}\theta_{i}\bigg\}.
\end{align*}
It then follows from Lemma 1.1 in \cite{GG1982} that for any $0<\tau<1$ and $\gamma>0$,
\begin{align*}
\|u\|_{L^{\infty}(B^{+}_{\tau})}\leq\frac{C}{(1-\tau)^{\beta}}\|u\|_{L^{\gamma}(B^{+}_{1},w)}+C\mathcal{H}.
\end{align*}
The proof is complete.

\end{proof}

\subsection{Local regularity for weak solutions}
To complete the proofs of Theorems \ref{ZWTHM90} and \ref{THM002}, the key is to establish the following two De Giorgi lemmms including the improvement on local oscillation of the solution in Lemma \ref{LEM0035ZZW} and the explicit decay estimates for the distribution function of the solution in Lemma \ref{lem005ZZW}.
\begin{lemma}\label{LEM0035ZZW}
Assume as in Theorems \ref{ZWTHM90} and \ref{THM002}. For $R\in(0,1)$, let $0\leq\sup\limits_{B_{R}^{+}}u\leq M\leq\|u\|_{L^{\infty}(B_{1}^{+})}$. Then there exists a small constant $0<\tau_{0}<1$, depending only on $n,m,p,q,\theta_{1},\theta_{2},\theta_{3}$, such that for $0\leq k<M$ and $0<\tau<\tau_{0}$, if
\begin{align}\label{E90}
\delta:=M-k\geq\mathcal{H}R^{1-\frac{n+\sum^{3}_{i=1}\theta_{i}}{mp}},
\end{align}
and
\begin{align}\label{ZWZ007ZZW}
|\{x\in B_{R}^{+}: u>M-\delta\}|_{\mu}\leq\tau|B_{R}^{+}|_{\mu},
\end{align}
then
\begin{align}\label{DZ001ZZW}
u\leq M-\frac{\delta}{2},\quad\mathrm{in}\; B^{+}_{R/2}.
\end{align}

\end{lemma}

\begin{proof}
For $\delta>0$ and $i=0,1,2,...,$ denote
\begin{align*}
r_{i}=\frac{R}{2}+\frac{R}{2^{i+1}},\quad k_{i}=M-\delta+\frac{\delta}{2}(1-2^{-i}).
\end{align*}
Pick a cut-off function $\xi_{i}\in C_{0}^{\infty}(B_{r_{i}})$ satisfying that $\xi_{i}=1\;\mathrm{in}\;B^{+}_{r_{i+1}}$, $0\leq\xi_{i}\leq1,$ and $|\nabla\xi_{i}|\leq C(r_{i}-r_{i+1})^{-1}$. For $k\in[0,M]$ and $r\in(0,R]$, let $v_{i}=(u-k_{i})_{+}$ and $A(k,r)=\{x\in B^{+}_{r}: u>k\}$. Carry out zero extension for $u$ in $B_{1}\setminus B_{1}^{+}$ and still use $u$ to represent the extended function. From Corollary \ref{CO901} and Lemma \ref{lem003}, we obtain
\begin{align*}
&(k_{i+1}-k_{i})^{p}|A(k_{i+1},r_{i+1})|_{\mu}^{\frac{1}{\chi}}\leq\bigg(\int_{B_{R}}|\xi_{i}v_{i}|^{p\chi}w\bigg)^{\frac{1}{\chi}}\leq C\int_{B_{R}}|\nabla(\xi_{i}v_{i})|^{p}w\notag\\
&\leq C\bigg(\int_{B_{R}^{+}}|\nabla\xi_{i}|v_{i}^{p}w+\mathcal{H}^{p}|A(k_{i},r_{i})|_{\mu}^{1-\frac{1}{m}}\bigg)\notag\\
&\leq C\left(\frac{(M-k_{i})^{p}}{(r_{i}-r_{i+1})^{p}}|A(k_{i},r_{i})|_{\mu}+\mathcal{H}^{p}|A(k_{i},r_{i})|_{\mu}^{1-\frac{1}{m}}\right),
\end{align*}
where $\chi$ and $\mathcal{H}$ are, respectively, defined by \eqref{CHI} and \eqref{F09}. In light of \eqref{E90}, we have
\begin{align*}
|A(k_{i+1},r_{i+1})|_{\mu}\leq& C\left(\frac{4^{(i+2)p}}{R^{p}}|A(k_{i},r_{i})|_{\mu}+\frac{2^{(i+2)p}\mathcal{H}^{p}}{\delta^{p}}|A(k_{i},r_{i})|^{1-\frac{1}{m}}_{\mu}\right)^{\chi}\notag\\
\leq&\left(\frac{C2^{(i+2)p}}{R^{p-\frac{n+\sum^{3}_{i=1}\theta_{i}}{m}}}|A(k_{i},r_{i})|^{1-\frac{1}{m}}_{\mu}\right)^{\chi}.
\end{align*}
Define
\begin{align*}
F_{i}:=\frac{|A(k_{i},r_{i})|_{\mu}}{|B^{+}_{R}|_{\mu}}.
\end{align*}
Then we have
\begin{align*}
F_{i+1}\leq&(C4^{(i+2)p})^{\chi} F_{i}^{\frac{\chi(m-1)}{m}}\leq\prod\limits^{i}_{s=0}\left[(C4^{p(i+2-s)})^{\chi}\right]^{\left(\frac{\chi(m-1)}{m}\right)^{s}}F_{0}^{\left(\frac{\chi(m-1)}{m}\right)^{i+1}}\notag\\
\leq&(C^{\ast}F_{0})^{\left(\frac{\chi(m-1)}{m}\right)^{i+1}}.
\end{align*}
Take $\tau_{0}=(C^{\ast})^{-1}$. Hence, it follows that if \eqref{ZWZ007ZZW} holds for any $0<\tau<\tau_{0}$, then $C^{\ast}F_{0}<1$ and $F_{i+1}\rightarrow0$, as $i\rightarrow\infty$. Therefore, \eqref{DZ001ZZW} holds.

\end{proof}

Before establishing the decaying estimates for the distribution function of the solution $u$, we first give the required anisotropic weighted isoperimetric inequalities. Similar to the proof of Proposition 2.10 in \cite{MZ2023}, we have from Lemma \ref{QWZM090} that
\begin{lemma}\label{prop002}
Let $n\geq2$ and $1<p<\infty$. If $(\theta_{1},\theta_{2},\theta_{3})\in[(\mathcal{A}\cup\mathcal{B})\cap(\mathcal{C}_{p}\cup\mathcal{D}_{p})]\cup\{\theta_{1}=\theta_{3}=0,\,\theta_{2}\geq n(p-1)\}$, then there exists some constant $1<\tilde{p}=\tilde{p}(n,p,\theta_{1},\theta_{2},\theta_{3})<p$ such that for any $R>0$, $l>k$ and $u\in W^{1,\tilde{p}}(B_{R},w)$,
\begin{align*}
&(l-k)^{\tilde{p}}\bigg(\int_{\{u\geq l\}\cap B_{R}}wdx\bigg)^{\tilde{p}}\int_{\{u\leq k\}\cap B_{R}}wdx\notag\\
&\leq C(n,p,\theta_{1},\theta_{2},\theta_{3})R^{\tilde{p}(n+\theta_{1}+\theta_{2}+\theta_{3}+1)}\int_{\{k<u<l\}\cap B_{R}}|\nabla u|^{\tilde{p}}wdx,
\end{align*}
and
\begin{align*}
&(l-k)^{\tilde{p}}\bigg(\int_{\{u\leq k\}\cap B_{R}}wdx\bigg)^{\tilde{p}}\int_{\{u\geq l\}\cap B_{R}}wdx\notag\\
&\leq C(n,p,\theta_{1},\theta_{2},\theta_{3})R^{\tilde{p}(n+\theta_{1}+\theta_{2}+\theta_{3}+1)}\int_{\{k<u<l\}\cap B_{R}}|\nabla u|^{\tilde{p}}wdx,
\end{align*}
where $w=|x'|^{\theta_{1}}|x|^{\theta_{2}}|x_{n}|^{\theta_{3}}.$
\end{lemma}
\begin{remark}
If $(\theta_{1},\theta_{2},\theta_{3})\in\mathcal{F}_{p_{0}}\cup\mathcal{G}_{p_{0}}$, applying the proof of Proposition 2.10 in \cite{MZ2023} again, it follows from \eqref{A095} that for any $R>0$, $l>k$ and $u\in W^{1,1}(B_{R})$,
\begin{align*}
&(l-k)\int_{\{u\geq l\}\cap B_{R}}wdx\int_{\{u\leq k\}\cap B_{R}}wdx\notag\\
&\leq C(n,p_{0},\theta_{1},\theta_{2},\theta_{3})R^{n+2(\theta_{1}+\theta_{2}+\theta_{3})+1}\int_{\{k<u<l\}\cap B_{R}}|\nabla u|dx.
\end{align*}
This type of anisotropic weighted isoperimetric inequality can simplify the computations in solving the similar decaying estimates in the next lemma. We leave it to interested readers.
\end{remark}

Using Lemma \ref{prop002}, we now present the desired decaying estimates as follows.
\begin{lemma}\label{lem005ZZW}
Assume as in Theorems \ref{ZWTHM90} and \ref{THM002}. For $R\in(0,1/2)$, let $0<\gamma<1$ and $0\leq\sup\limits_{B^{+}_{2R}}u\leq M\leq\|u\|_{L^{\infty}(B^{+}_{1})}$. Then there exists some constant $1<\tilde{p}=\tilde{p}(n,p,\theta_{1},\theta_{2},\theta_{3})<p$ such that for every $j\geq0$, we have either
\begin{align}\label{VD008}
M\leq2^{j}\mathcal{H}R^{1-\frac{n+\sum^{3}_{i=1}\theta_{i}}{mp}},
\end{align}
or
\begin{align}\label{DEC001ZZW}
\frac{|\{x\in B^{+}_{R}:u>M-\frac{M}{2^{j}}\}|_{\mu}}{|B^{+}_{R}|_{\mu}}\leq\frac{C_{0}}{j^{\frac{p-\tilde{p}}{p\tilde{p}}}},
\end{align}
where $C_{0}=C_{0}(n,m,p,\theta_{1},\theta_{2},\theta_{3})$.

\end{lemma}

\begin{proof}
Perform zero extension for $u$ in $B_{1}\setminus B_{1}^{+}$ and still write the extended function as $u$. For $i\geq0$, set $k_{i}=M-\frac{M}{2^{i}}$ and $A(k_{i},R)=B_{R}\cap\{u>k_{i}\}.$ Using Lemma \ref{prop002}, we see that for some constant $1<\tilde{p}=\tilde{p}(n,p,\theta_{1},\theta_{2},\theta_{3})<p$,
\begin{align}\label{WMZ001ZZW}
&(k_{i+1}-k_{i})^{\tilde{p}}|A(k_{i+1},R)|_{\mu}^{\tilde{p}}|B_{R}\setminus A(k_{i},R)|_{\mu}\notag\\
&\leq CR^{\tilde{p}(n+\sum^{3}_{i=1}\theta_{i}+1)}\int_{A(k_{i},R)\setminus A(k_{i+1},R)}|\nabla u|^{\tilde{p}}wdx.
\end{align}
Note that
\begin{align*}
|B_{R}\setminus A(k_{i},R)|_{\mu}\geq|B_{R}\setminus B^{+}_{R}|_{\mu}\geq C R^{n+\sum^{3}_{i=1}\theta_{i}}.
\end{align*}
Substituting this into \eqref{WMZ001ZZW}, we derive
\begin{align*}
&|A(k_{i+1},R)|_{\mu}\leq\frac{C2^{i+1}}{M}R^{\frac{(n+\theta_{1}+\theta_{2}+\theta_{3})(\tilde{p}-1)}{\tilde{p}}+1}\bigg(\int_{A(k_{i},R)\setminus A(k_{i+1},R)}|\nabla u|^{\tilde{p}}wdx\bigg)^{\frac{1}{\tilde{p}}}.
\end{align*}
In view of $1<\tilde{p}<p$, it follows from H\"{o}lder's inequality that
\begin{align*}
&\int_{A(k_{i},R)\setminus A(k_{i+1},R)}|\nabla u|^{\tilde{p}}wdx\notag\\
&\leq\bigg(\int_{A(k_{i},R)\setminus A(k_{i+1},R)}|\nabla u|^{\tilde{p}}wdx\bigg)^{\frac{\tilde{p}}{p}}\bigg(\int_{A(k_{i},R)\setminus A(k_{i+1},R)}wdx\bigg)^{\frac{p-\tilde{p}}{p}}\notag\\
&\leq\bigg(\int_{B^{+}_{R}}|\nabla(u-k_{i})_{+}|^{p}wdx\bigg)^{\frac{\tilde{p}}{p}}|A(k_{i},R)\setminus A(k_{i+1},R)|_{\mu}^{\frac{p-\tilde{p}}{p}}.
\end{align*}
Choose a cut-off function $\xi\in C_{0}^{\infty}(B_{2R})$ satisfying that
\begin{align*}
\xi=1\;\mathrm{in}\;B_{R},\quad0\leq\xi\leq1,\;|\nabla\xi|\leq\frac{C(n)}{R}\;\,\mathrm{in}\;B_{2R}.
\end{align*}
A consequence of Lemma \ref{lem003} shows that
\begin{align*}
&\bigg(\int_{B^{+}_{R}}|\nabla (u-k_{i})_{+}|^{p}wdx\bigg)^{\frac{1}{p}}\notag\\
&\leq C\bigg(\int_{B^{+}_{2R}}|(u-k_{i})_{+}|^{p}|\nabla\xi|^{p}wdx+\mathcal{H}^{p}|A(k_{i},2R)|_{\mu}^{1-\frac{1}{m}}\bigg)^{\frac{1}{p}}\notag\\
&\leq C\bigg(\frac{M^{p}}{2^{ip}}R^{n+\sum^{3}_{i=1}\theta_{i}-p}+\mathcal{H}^{p}R^{(n+\sum^{3}_{i=1}\theta_{i})(1-\frac{1}{m})}\bigg).
\end{align*}
If \eqref{VD008} fails for some $j$, it then follows that for $i\leq j$,
\begin{align*}
\bigg(\int_{B^{+}_{R}}|\nabla (u-k_{i})_{+}|^{p}wdx\bigg)^{\frac{1}{p}}\leq\frac{CM}{2^{i}}R^{\frac{n+\sum^{3}_{i=1}\theta_{i}-p}{p}}.
\end{align*}
We have from these above facts that
\begin{align*}
|A(k_{i+1},R)|_{\mu}\leq&CR^{(n+\sum^{3}_{i=1}\theta_{i})(1-\frac{p-\tilde{p}}{p\tilde{p}})}|A(k_{i},R)\setminus A(k_{i+1},R)|_{\mu}^{\frac{p-\tilde{p}}{p\tilde{p}}}.
\end{align*}
This leads to that for $j\geq1$,
\begin{align*}
j|A(k_{j},R)|_{\mu}^{\frac{p\tilde{p}}{p-\tilde{p}}}\leq&\sum^{j-1}_{i=0}|A(k_{i+1},R)|_{\mu}^{\frac{p\tilde{p}}{p-\tilde{p}}}\notag\\
\leq& CR^{(n+\sum^{3}_{i=1}\theta_{i})(\frac{p\tilde{p}}{p-\tilde{p}}-1)}|B_{R}|_{\mu}\leq C|B_{R}|_{\mu}^{\frac{p\tilde{p}}{p-\tilde{p}}}.
\end{align*}
Then \eqref{DEC001ZZW} is proved.

\end{proof}

We are now ready to prove Theorems \ref{ZWTHM90} and \ref{THM002} in turn.
\begin{proof}[Proof of Theorem \ref{ZWTHM90}]
For $0<R\leq\frac{1}{2}$, denote
\begin{align*}
\overline{\mu}(R)=\sup_{B^{+}_{R}}u,\quad\underline{\mu}(R)=\inf_{B^{+}_{R}}u,\quad \omega(R)=\overline{\mu}(R)-\underline{\mu}(R).
\end{align*}
Take a sufficiently large integer $j_{0}$ such that $C_{0}j_{0}^{-\frac{p-\tilde{p}}{p\tilde{p}}}\leq\tau_{0},$ where $\tau_{0}$ and $C_{0}$ are, respectively, given by Lemma \ref{LEM0035ZZW} and \eqref{DEC001ZZW}. It then follows from Lemmas \ref{LEM0035ZZW} and \ref{lem005ZZW} that there holds
\begin{align*}
\mathrm{either}\;\;\overline{\mu}(R)\leq2^{j_{0}}\mathcal{H}R^{1-\frac{n+\sum^{3}_{i=1}\theta_{i}}{mp}},\quad\mathrm{or}\;\;\overline{\mu}(R/4)\leq\overline{\mu}(R)-\frac{\overline{\mu}(R)}{2^{j_{0}+1}}.
\end{align*}
Using the above arguments for $-u$, we obtain
\begin{align*}
\mathrm{either}\;\;\underline{\mu}(R)\geq-2^{j_{0}}\mathcal{H}R^{1-\frac{n+\sum^{3}_{i=1}\theta_{i}}{mp}},\quad\mathrm{or}\;\;\underline{\mu}(R/4)\geq\underline{\mu}(R)-\frac{\underline{\mu}(R)}{2^{j_{0}+1}}.
\end{align*}
In either case, we all have
\begin{align*}
\omega(R/4)\leq(1-2^{-j_{0}-1})\omega(R)+2^{j_{0}+1}\mathcal{H}R^{1-\frac{n+\sum^{3}_{i=1}\theta_{i}}{mp}}=\frac{1}{4^{\beta_{1}}}\omega(R)+2^{j_{0}+1}\mathcal{H}R^{\beta_{2}},
\end{align*}
where $\beta_{1}=-\frac{\ln(1-2^{-j_{0}-1})}{\ln4}$ and $\beta_{2}=1-\frac{n+\sum^{3}_{i=1}\theta_{i}}{mp}$. For every $0<R\leq\frac{1}{2}$, there exists an integer $k\geq1$ such that $4^{-(k+1)}\cdot2^{-1}<R\leq4^{-k}\cdot2^{-1}$. Due to the fact that $\omega(R)$ is nondecreasing in $R$, it follows from Theorem \ref{THE621} that for any $\gamma>0$,
\begin{align*}
\omega(R)\leq&4^{-k\beta_{1}}\omega(2^{-1})+2^{j_{0}+1-\beta_{2}}4^{(1-k)\beta_{2}}\mathcal{H}\sum^{k-1}_{i=0}4^{i(\beta_{2}-\beta_{1})}\notag\\
\leq& C(\|u\|_{L^{\gamma}(B^{+}_{1},w)}+\mathcal{H})R^{\min\{\beta_{1},\beta_{2}\}},
\end{align*}
where $C=C(n,m,p,q,\gamma,\theta_{1},\theta_{2},\theta_{3}).$

\end{proof}

\begin{proof}[Proof of Theorem \ref{THM002}]
Take the proof corresponding to the case of $\theta_{1}\neq0$ and $\theta_{2}=\theta_{3}=0$ for example. The other two cases are the same and thus omitted. Applying the proof of Theorem \ref{ZWTHM90} with a slight modification, we obtain that there exists two constants $0<\alpha<1$ and $C>1$, depending only on $n,m,p,q,\theta_{1},$ such that for any fixed $\bar{x}\in\{(0',\bar{x}_{n}):0\leq\bar{x}_{n}\leq1/8\}$,
\begin{align}\label{QM916}
|u(x)-u(\bar{x})|\leq C(\|\phi_{0}\|_{L^{\infty}(\partial''B_{1}^{+})}+\mathcal{H})|x|^{\alpha},\quad\text{for all }x\in B^{+}_{1/4}(\bar{x}).
\end{align}
For $R\in(0,1/8)$ and $y\in B^{+}_{1/R}$, let $u_{R}(y)=u(Ry)$, $f_{0,R}(y)=f_{0}(Ry)$ and $f_{1,R}(y)=f_{1}(Ry)$. Then $u_{R}$ solves
\begin{align*}
-\mathrm{div}(|y'|^{\theta_{1}}|\nabla u_{R}|^{p-2}\nabla u_{R})=R^{p-\theta_{1}}f_{0,R}+R^{p}|y'|^{\theta_{1}}f_{1,R},\quad\mathrm{in}\;B^{+}_{1/R}.
\end{align*}

For any given $x,\tilde{x}\in B^{+}_{1/8},$ assume without loss of generality that $|\tilde{x}'|\leq|x'|$. Let $R=|x'|$. Then using Theorem \ref{CORO06} and the interior H\"{o}lder estimates for degenerate elliptic equations, we deduce that there exist two constants $0<\beta<1$ and $C>1$, both depending only on $n,m,p,q,\theta_{1}$, such that for any fixed $\bar{y}\in B^{+}_{1/(4R)}\cap\{|y'|=1\}$,
\begin{align}\label{WAQA001}
|u_{R}(y)-u_{R}(\bar{y})|\leq C(\|\phi_{0}\|_{L^{\infty}(\partial''B_{1}^{+})}+\mathcal{H})|y-\bar{y}|^{\beta},\quad\text{for every }y\in B_{1/2}^{+}(\bar{y}).
\end{align}
Note that
\begin{align*}
|u(x)-u(\tilde{x})|\leq|u(x',x_{n})-u(x',\tilde{x}_{n})|+|u(x',\tilde{x}_{n})-u(\tilde{x}',\tilde{x}_{n})|.
\end{align*}
Set $\lambda>1$. If $|x_{n}-\tilde{x}_{n}|\leq R^{\lambda}$, it follows from \eqref{WAQA001} that
\begin{align*}
&|u(x',x_{n})-u(x',\tilde{x}_{n})|\leq\left|u_{R}(x'/R,x_{n}/R)-u_{R}(x'/R,\tilde{x}_{n}/R)\right|\notag\\
&\leq C(\|\phi_{0}\|_{L^{\infty}(\partial''B_{1}^{+})}+\mathcal{H})|(x_{n}-\tilde{x}_{n})/R|^{\beta}\notag\\
&\leq C(\|\phi_{0}\|_{L^{\infty}(\partial''B_{1}^{+})}+\mathcal{H})|x_{n}-\tilde{x}_{n}|^{\frac{(\lambda-1)\beta}{\lambda}}.
\end{align*}
By contrast, if $|x_{n}-\tilde{x}_{n}|>R^{\lambda}$, using \eqref{QM916}, we obtain
\begin{align*}
&|u(x',x_{n})-u(x',\tilde{x}_{n})|\notag\\
&\leq|u(x',x_{n})-u(0',x_{n})|+|u(0',x_{n})-u(0',\tilde{x}_{n})|+|u(0',\tilde{x}_{n})-u(x',\tilde{x}_{n})|\notag\\
&\leq C(\|\phi_{0}\|_{L^{\infty}(\partial''B_{1}^{+})}+\mathcal{H})(|x'|^{\alpha}+|x_{n}-\tilde{x}_{n}|^{\alpha})\notag\\
&\leq C(\|\phi_{0}\|_{L^{\infty}(\partial''B_{1}^{+})}+\mathcal{H})|x_{n}-\tilde{x}_{n}|^{\frac{\alpha}{\lambda}}.
\end{align*}
Similarly, if $|x'-\tilde{x}'|\leq R^{\lambda}$, we have from \eqref{WAQA001} that
\begin{align*}
&|u(x',\tilde{x}_{n})-u(\tilde{x}',\tilde{x}_{n})|=\left|u_{R}(x'/R,\tilde{x}_{n}/R)-u_{R}(\tilde{x}'/R,\tilde{x}_{n}/R)\right|\notag\\
&\leq C(\|\phi_{0}\|_{L^{\infty}(\partial''B_{1}^{+})}+\mathcal{H})|(x'-\tilde{x}')/R|^{\beta}\notag\\
&\leq C(\|\phi_{0}\|_{L^{\infty}(\partial''B_{1}^{+})}+\mathcal{H})|x'-\tilde{x}'|^{\frac{(\lambda-1)\beta}{\lambda}},
\end{align*}
while, if $|x'-\tilde{x}'|>R^{\lambda}$, it follows from \eqref{QM916} that
\begin{align*}
&|u(x',\tilde{x}_{n})-u(\tilde{x}',\tilde{x}_{n})|\leq|u(x',\tilde{x}_{n})-u(0',\tilde{x}_{n})|+|u(0',\tilde{x}_{n})-u(\tilde{x}',\tilde{x}_{n})|\notag\\
&\leq C(\|\phi_{0}\|_{L^{\infty}(\partial''B_{1}^{+})}+\mathcal{H})\big(R^{\alpha}+|\tilde{x}'|^{\alpha}\big)\leq C(\|\phi_{0}\|_{L^{\infty}(\partial''B_{1}^{+})}+\mathcal{H})|x'-\tilde{x}'|^{\frac{\alpha}{\lambda}}.
\end{align*}
Combining these above facts, we choose $\lambda=1+\frac{\alpha}{\beta}$ such that $\frac{(\lambda-1)\beta}{\lambda}=\frac{\alpha}{\lambda}$. In fact, this choice is best, since it achieves the maximum H\"{o}lder regularity exponent in virtue of the fact that $\frac{\lambda-1}{\lambda}$ is increasing in $\lambda$ and $\frac{1}{\lambda}$ is decreasing in $\lambda$. Therefore, we obtain that for any $x,\tilde{x}\in B_{1/8}^{+},$
\begin{align*}
|u(x)-u(\tilde{x})|\leq C(\|\phi_{0}\|_{L^{\infty}(\partial''B_{1}^{+})}+\mathcal{H})|x-\tilde{x}|^{\frac{\alpha\beta}{\alpha+\beta}},
\end{align*}
which implies that Theorem \ref{THM002} holds.

\end{proof}

\section{Further discussions and some open problems}\label{SEC05}

Let $\Omega$ be a bounded open set in $\mathbb{R}^{n}$. Denote $\Omega_{T}:=\Omega\times(-T,0]$, $T>0$. In this section, we mainly focus on the nonnegative solutions for doubly non-linear parabolic equations with anisotropic weights as follows:
\begin{align}\label{PROBLEM06}
\begin{cases}
w_{1}\partial_{t}u^{q}-\mathrm{div}(w_{2}|\nabla u|^{p-2}\nabla u)=0,& \mathrm{in}\;\Omega_{T},\\
u=0,&\mathrm{on}\;\partial\Omega\times(-T,0],\\
u(x,0)=\phi_{0}\geq0,&
\end{cases}
\end{align}
where $q>0,\,p>1$, $\phi_{0}\in L^{\infty}(\Omega)$, $w_{1}=|x'|^{\theta_{1}}|x|^{\theta_{2}}|x_{n}|^{\theta_{3}}$ and $w_{2}=|x'|^{\theta_{4}}|x|^{\theta_{5}}|x_{n}|^{\theta_{6}}$. A nonnegative function $u$ is said to be a weak solution of \eqref{PROBLEM06} if
$$u\in C([0,T];L^{q+1}(\Omega,w_{1}))\cap L^{p}(0,T;W_{0}^{1,p}(\Omega,w_{2}))$$
satisfies
\begin{align*}
\int_{\Omega}w_{1}u^{q}\varphi dx\Big|_{t_{1}}^{t_{2}}+\int_{t_{1}}^{t_{2}}\int_{\Omega}(-w_{1}u^{q}\partial_{t}\varphi+w_{2}|\nabla u|^{p-2}\nabla u\nabla\varphi)dxdt=0,
\end{align*}
for any $-T\leq t_{1}<t_{2}\leq0$ and $\varphi\in W^{1,q+1}(0,T;L^{q+1}(\Omega,w_{1}))\cap L^{p}(0,T;W_{0}^{1,p}(\Omega,w_{2}))$. The equation in \eqref{PROBLEM06} is frequently utilized to describe various diffusion phenomena occurring in heat flow, chemical concentration, gas-kinetics, plasmas and thin liquid film dynamics. The problem of precisely analyzing the diffusion processes of the flows is critical to their applications in industry. For more related physical interpretations and derivations of physical models, we refer to \cite{DK2007,V2007,BDGLS2023} and the references therein.

According to the diffusion mechanism of the flows, the equation in \eqref{PROBLEM06} can be classified into the following three types. If $q>p-1$, it is called fast diffusion equation and its solution will extinct in finite time. By contrast, the case of $q<p-1$ corresponds to slow diffusion equation and its solution always decays with power-function to the stable state. The equation in borderline case $q=p-1$ is called Trudinger's equation, which was originally observed by Trudinger \cite{T1968}. Its solution decays faster than slow diffusion equation and decreases exponentially in time variable. Moreover, there is a special feature for Trudinger's equation that it can be linked to extremals of Poincar\'{e} inequalities, see \cite{LL2022} for more details. As for the existence of weak solutions to problem \eqref{PROBLEM06}, see Theorem 1.7 in \cite{AL1983}, Theorem 1.3 in \cite{BDMS2018} and Theorem 5.1 in \cite{BDMS201802}. The uniqueness result can be seen in Corollary 4.17 of \cite{BDGLS2023}. With regard to the study on H\"{o}lder regularity of weak solutions to problem \eqref{PROBLEM06} without weights, the well-studied ranges contain the following two cases: one is that $q>p-1$ and $1<p\leq2$; another is that $q\leq p-1$ and $p\geq2$. Especially when the considered diffusion operator is linear, there has a long list of literature devoted to studying the regularity (see e.g. \cite{JX2019,JX2022,S1983,K1988,BV2010,DK2007,DK1992,DGV2012,DKV1991,JRX2023,BFR2017,BFV2018}) and the asymptotic behavior  (see e.g. \cite{BH1980,BGV2012,FS2000,BF2021,A2021,JRX2023,AP1981,BSV2015}). When the diffusion operator is nonlinear, see \cite{BDGLS2023} for systematic introduction on doubly non-linear parabolic equations without weights.

However, to the best of our knowledge, less work is dedicated to investigating the weighted cases in \eqref{PROBLEM06}. When the diffusion operate is linear, that is, when $p=2$, we refer to \cite{MZ2023,JX2022} for weighted fast diffusion equations and \cite{JRX2023} for weighted porous medium equations. With regard to the case when $q=1$ and $p>2$, see \cite{MZ2024}. From Theorem 3.11 in \cite{JX2022}, Theorem 1.6 in \cite{MZ2023}, Theorem 2.10 in \cite{JRX2023} and Theorem 1.1 in \cite{MZ2024}, we see that the decay rate exponent $\alpha$ of the solution near the degenerate or singular points of the weights is just known to be in $(0,1)$ but its value is not explicit. Therefore, a natural problem is to solve the optimal decay rate of the solution near these degenerate or singular points of the weights. In fact, we can use the standard separation of variables method to obtain exact solution for heat conduction equation with monomial weight $|x|^{\theta}$ in cylindrical domain. The value of $\theta$ can affect the diffusion rate in time variable with either enhancement or reduction. Based on these above facts, the problem of accurately quantifying the enhancement and reduction effect induced by weighted diffusion operators may have a potential application in manufacturing porous medium materials with special permeation rates and controlling diffusion processes according to the needs of the industry.

Denote
\begin{align*}
&d\mu_{1}=w_{1}dx=|x'|^{\theta_{1}}|x|^{\theta_{2}}|x_{n}|^{\theta_{3}}dx,\quad d\mu_{2}=w_{2}dx=|x'|^{\theta_{4}}|x|^{\theta_{5}}|x_{n}|^{\theta_{6}}dx, \notag\\ &d\mu_{3}=w_{3}dx=|x'|^{\theta_{7}}|x|^{\theta_{8}}|x_{n}|^{\theta_{9}}dx.
\end{align*}
Based on the investigations in this paper, we now sum up the following three intriguing problems for readers' convenience.

{\bf P1.} Find the optimal ranges of parameters $s,p,q,\beta$ and $\theta_{i}$, $i=1,2,...,9$ such that the following interpolation inequality holds: for any $u\in C^{\infty}_{0}(\mathbb{R}^{n})$,
\begin{align*}
\|u\|_{L^{s}(\mathbb{R}^{n},w_{1})}\leq C\|\nabla u\|^{\beta}_{L^{p}(\mathbb{R}^{n},w_{2})}\|u\|^{1-\beta}_{L^{q}(\mathbb{R}^{n},w_{3})},
\end{align*}
where the constant $C$ depends only on $n,s,p,q,\beta$ and $\theta_{i}$, $i=1,2,...,9$.

{\bf P2.} Find the optimal ranges of parameters $p,\gamma$ and $\theta_{i}$, $i=1,2,...,6$ such that the following anisotropic weighted Poincar\'{e} inequality holds: for any $u\in W^{1,p}(B_{R},w_{2})$, $R>0$,
\begin{align*}
\int_{B_{R}}\bigg|u-\frac{1}{|B_{R}|_{\mu_{1}}}\int_{B_{R}}ud\mu_{1}\bigg|^{p}d\mu_{1}\leq CR^{\gamma}\int_{B_{R}}|\nabla u|^{p}d\mu_{2},
\end{align*}
for some constant $C$ depending only on $n,p$ and $\theta_{i},$ $i=1,2,...,6$.

{\bf P3.} Solve the optimal decay rate of the solution to problem \eqref{PROBLEM06} near the singular or degenerate points of the weights and meanwhile classify the enhancement and reduction effect on the diffusion arising from anisotropic weights.


\noindent{\bf{\large Acknowledgements.}} This work was supported in part by the National Key research and development program of
China (No. 2022YFA1005700 and 2020YFA0712903). C. Miao was partially supported by the National Natural Science Foundation of China (No. 12371095 and 12071043).



\end{document}